\documentclass[11pt]{article}
\usepackage{amsfonts,amscd,amssymb, amsmath}
\usepackage{GrCalc3}
\newtheorem{definition}{Definition}[section]
\newtheorem{lemma}[definition]{Lemma}
\newtheorem{proposition}[definition]{Proposition}

\newtheorem{theorem}[definition]{Theorem}

\def\va{\varepsilon}

\def\ra{\rightarrow}
\def\a{\alpha}
\def\b{\beta}

\def\l{\lambda}
\def\r{\rho}
\def\cd{\cdot}

\def\mf{\mathfrak}
\def\mb{\mathbb}

\newcommand{\hide}[1]{}

\newcommand{\inv}{^{-1}} 
\newcommand{\C}{\mathcal{C}}
\newcommand{\yd}{{}_H^H\mathcal{YD}}
\newcommand{\ydc}{{}_H^H\mathcal{YD}(\C)}
\newcommand{\CC}{{}_H^H\C_H^H}
\newcommand{\gnotc}[1]{\gnot{\hspace{4mm}#1}}
\newcommand{\morph}{\xi}

\def\rawo\lonra{\longrightarrow}

\def\ot{\otimes}

\allowdisplaybreaks[4]

\newenvironment{proof}{{\it Proof.}}{\hfill $ \square $ \vskip 4mm}
\begin{document}
\title{Structure theorems for bicomodule algebras over quasi-Hopf algebras, weak Hopf algebras 
and braided Hopf algebras
\thanks
{Research partially supported by FWO-Vlaanderen (Flemish Fund for Scientific Research) 
within the research project ''Equivariant Brauer groups and Galois deformations''.  
The first author (J. D.) was supported as an aspirant of FWO-Vlaanderen. 
The second author (F. P.) was also partially supported by 
a grant of the Romanian National 
Authority for Scientific Research, CNCS-UEFISCDI, 
project number PN-II-ID-PCE-2011-3-0635,  
contract nr. 253/5.10.2011.}}
\author{Jeroen Dello\\
Department of Mathematics and Statistics, University of Hasselt\\
Agoralaan 1, Diepenbeek 3590, Belgium\\
e-mail: jeroen.dello@uhasselt.be
\and
Florin Panaite\\
Institute of Mathematics of the 
Romanian Academy\\ 
PO-Box 1-764, RO-014700 Bucharest, Romania\\
e-mail: florin.panaite@imar.ro
\and 
Freddy Van Oystaeyen\\
Department of Mathematics and Computer Science, University of Antwerp\\
Middelheimlaan 1, Antwerp 2020, Belgium\\
e-mail: fred.vanoystaeyen@ua.ac.be
\and
Yinhuo Zhang \\
Department of Mathematics and Statistics, University of Hasselt\\
Agoralaan 1, Diepenbeek 3590, Belgium\\
e-mail: yinhuo.zhang@uhasselt.be
}
\date{}
\maketitle

\begin{abstract}
Let $H$ be a quasi-Hopf algebra, a weak Hopf algebra 
or a braided Hopf algebra. Let $B$ be an $H$-bicomodule algebra 
such that there exists a morphism of $H$-bicomodule 
algebras $v:H\rightarrow B$. Then we can define an object $B^{co(H)}$ which is a left-left 
Yetter-Drinfeld module over $H$, having extra properties that allow to make a smash product 
$B^{co(H)}\# H$ which is an $H$-bicomodule algebra, isomorphic to $B$.
\end{abstract}
\section*{Introduction}
${\;\;\;}$The need for a structure theorem for bicomodule algebras (over a classical Hopf algebra $H$) 
finds its  
origin in the doctoral thesis of the first author. That is to say, in \cite{Dello} the author constructed a 
morphism $\morph : BiGal(\yd;B) \rightarrow BiGal(B \rtimes H)$ from the group of braided bi-Galois objects 
over a braided Hopf algebra $B$ in the category $\yd$ (of left-left Yetter-Drinfeld modules), to the group 
of bi-Galois objects over the Radford biproduct $B \rtimes H$. This morphism $\morph$ sends an 
isomorphism class $[A]$ to the class $[A\# H]$. Using the structure theorem for bicomodule algebras, 
one is able to give a description for the image of this morphism $\morph$.

To be more specific, (the isomorphism class of) a $B \rtimes H$-bi-Galois object $D$ is an element of 
$Im\morph$ if there exists a $B \rtimes H$-bicolinear algebra morphism $v : H \rightarrow D$. Indeed, 
under this assumption $v : H \rightarrow D$ is also an $H$-bicolinear algebra morphism. Hence, 
by the structure theorem, we get a Yetter-Drinfeld module algebra $D_0 = D^{coH}$ such that 
$D \cong D_0 \# H$ as $H$-bicomodule algebras. Then it is shown in \cite{Dello} that $D_0$ 
has the structure of a braided $B$-bi-Galois object in $\yd$ and that $D \cong D_0 \# H$ as 
$B \rtimes H$-bicomodule algebras, thus realizing $\morph([D_0])=[D]$.

It appears natural to try to see whether the structure theorem for bicomodule algebras remains valid 
over more general Hopf algebra-type objects. It is the aim of this paper to prove that this happens 
if we replace the Hopf algebra $H$ by a quasi-Hopf algebra, a weak Hopf algebra 
or a braided Hopf algebra. In each case, the result is that if $B$ is an $H$-bicomodule algebra 
(in an appropriate sense in each case) such that there exists a morphism of $H$-bicomodule 
algebras $v:H\rightarrow B$, then we can define an object $B^{co(H)}$ which is a left-left 
Yetter-Drinfeld module over $H$, having extra properties that allow to make a smash product 
$B^{co(H)}\# H$ which is an $H$-bicomodule algebra, isomorphic to $B$. In all three cases, 
the proof relies on an analogue of Schauenburg's theorem that the categories of two-sided two-cosided 
Hopf modules and Yetter-Drinfeld modules are equivalent, cf. \cite{schauen}. 

The paper is organized as follows: it contains three sections, each one with its own preliminaries, 
each one containing the proof of the structure theorem for each of the three Hopf algebra-type 
objects mentioned above. 

In the first two sections, 
we work over a base field $k$. All algebras, linear spaces 
etc. are over $k$; unadorned $\ot $ means $\ot_k$. A multiplication $\mu :A\ot A\rightarrow  A$ on a 
linear space $A$ is denoted by juxtaposition: $\mu (a\ot a')=aa'$. For a comultiplication 
$\Delta :C\rightarrow C\ot C$ on a linear space $C$, we use the version of Sweedler's sigma notation: 
$\Delta (c)=c_1\ot  
c_2$, for $c\in C$. \\[2mm]  
{\bf Acknowledgements:} We would like to thank Bojana Femi\'c for providing us with the 
package (GrCalc3) to draw braided diagrams, which was developed by Bodo Pareigis. 
\section{Quasi-Hopf bicomodule algebras}\label{sec1}
\setcounter{equation}{0}
${\;\;\;}$Following Drinfeld \cite{d}, a quasi-bialgebra is a fourtuple $(H, \Delta ,
\va , \Phi )$, where $H$ is an associative algebra with unit $1$,  
$\Phi$ is an invertible element in $H\ot H\ot H$, and $\Delta :\
H\ra H\ot H$ and $\va :\ H\ra k$ are algebra homomorphisms
satisfying the identities (for all $h\in H$): 
\begin{eqnarray}
&&(id \ot \Delta )(\Delta (h))=%
\Phi (\Delta \ot id)(\Delta (h))\Phi ^{-1},\label{q1}\\
&&(id \ot \va )(\Delta (h))=h\ot 1, %
\mbox{${\;\;\;}$}%
(\va \ot id)(\Delta (h))=1\ot h,\label{q2} \\
&&(1\ot \Phi)(id\ot \Delta \ot id) (\Phi)(\Phi \ot 1)= (id\ot id
\ot \Delta )(\Phi ) (\Delta \ot id \ot id)(\Phi
),\label{q3}\\
&&(\va \ot id\ot id)(\Phi )=(id \ot \va \ot id )(\Phi )= 
(id \ot id\ot \va )(\Phi )=1\ot 1\ot 1.\label{q4}
\end{eqnarray}
The map $\Delta $ is called the coproduct or the
comultiplication, $\va $ the counit and $\Phi $ the associator.
We denote the tensor components of $\Phi$ 
by capital letters and those of $\Phi^{-1}$ by small letters: 
\begin{eqnarray*}
&&\Phi=X^1\ot X^2\ot X^3=T^1\ot T^2\ot T^3=Y^1\ot  
Y^2\ot Y^3=\cdots\\%
&&\Phi^{-1}=x^1\ot x^2\ot x^3=
t^1\ot t^2\ot t^3=y^1\ot y^2\ot y^3=\cdots
\end{eqnarray*}

The quasi-bialgebra $H$ is called a quasi-Hopf algebra if there exists an 
anti-automorphism $S$ of the algebra $H$ and elements $\a , \b \in
H$ such that, for all $h\in H$, we have:
\begin{eqnarray}
&&S(h_1)\a h_2=\va (h)\a \mbox{${\;\;\;}$ and ${\;\;\;}$}
h_1\b S(h_2)=\va (h)\b ,\label{q5}\\
&&X^1\b S(X^2)\a X^3=1 %
\mbox{${\;\;\;}$ and${\;\;\;}$}%
S(x^1)\a x^2\b S(x^3)=1.\label{q6}
\end{eqnarray}
The axioms for a quasi-Hopf algebra imply that $\va (\a )\va (\b 
)=1$, so, by rescaling $\a $ and $\b $, we may assume without loss
of generality that $\va (\a )=\va (\b )=1$ and $\va \circ S=\va $.

Suppose that $(H, \Delta , \varepsilon , \Phi )$ is a
quasi-bialgebra. If $U,V,W$ are left (right) $H$-modules, define
$a_{U,V,W}, {\bf a}_{U, V, W} :(U\otimes V)\otimes W\rightarrow
U\otimes (V\otimes W)$
by $a_{U,V,W}((u\otimes v)\otimes w)=\Phi \cdot (u\otimes
(v\otimes w))$ and 
${\bf a}_{U, V, W}((u\ot v)\ot w)= (u\ot (v\ot w))\cd \Phi ^{-1}$. 
The category $_H{\cal M}$ (${\cal M}_H$) of 
left (right) $H$-modules becomes a monoidal category (see
\cite{k}, \cite{m} for terminology) with tensor product 
$\otimes $ given via $\Delta $, associativity constraints
$a_{U,V,W}$ (${\bf a}_{U, V, W}$), unit $k$ as a trivial
$H$-module and the usual left and right
unit constraints.

Let again $H$ be a quasi-bialgebra. We say that a $k$-vector space 
$A$ is a left $H$-module algebra if it is an algebra in the
monoidal category $_H{\cal M}$, that is $A$ has a multiplication
and a usual unit $1_A$ satisfying the 
following conditions: 
\begin{eqnarray}
&&(aa')a''=(X^1\cd a)[(X^2\cd a')(X^3\cd
a'')],\label{ma1}\\
&&h\cd (a a')=(h_1\cd a)(h_2\cd a'), \;\;\;
h\cd 1_A=\va (h)1_A,
\end{eqnarray}
for all $a, a', a''\in A$ and $h\in H$, where $h\ot a\ra
h\cd a$ is the left $H$-module structure of $A$. Following
\cite{bpv} we define the smash product $A\# H$ as follows: as
vector space $A\# H$ is $A\ot H$ (elements $a\ot h$ will be
written $a\# h$) with multiplication
given by 
$(a\# h)(a'\# h')=
(x^1\cd a)(x^2h_1\cd a')\# x^3h_2h'$. 
The smash product $A\# H$ is an 
associative algebra with unit $1_A\# 1_H$.

Recall from \cite{hn1} the notions of (bi)comodule algebra over a
quasi-bialgebra.
\begin{definition}
Let $H$ be a quasi-bialgebra. A unital associative algebra
$\mathfrak{A}$ is called a right $H$-comodule algebra if there
exist an algebra morphism $\r :\mathfrak{A}\ra \mathfrak{A}\ot H$
and an invertible element $\Phi _{\r }\in \mathfrak{A}\ot H\ot H$
such that:
\begin{eqnarray}
&&\Phi _{\r }(\r \ot id)(\r (\mf {a}))=(id\ot \Delta
)(\r (\mf {a}))\Phi _{\r }, 
\mbox{${\;\;\;}$$\forall $ $\mf {a}\in
\mathfrak{A}$,}\label{rca1}\\[1mm]%
&&(1_{\mf {A}}\ot \Phi)(id\ot \Delta \ot id)(\Phi _{\r })(\Phi
_{\r }\ot 1_H)= (id\ot id\ot \Delta )(\Phi _{\r })(\r \ot id\ot
id)(\Phi _{\r }),\label{rca2}\\[1mm]%
&&(id\ot \va)\circ \r =id ,\label{rca3}\\[1mm]%
&&(id\ot \va \ot id)(\Phi _{\r })=(id\ot id\ot \va )(\Phi _{\r }
)=1_{\mathfrak{A}}\ot 1_H.\label{rca4}
\end{eqnarray}
Similarly, a unital associative algebra $\mathfrak{B}$ is called
a left $H$-comodule algebra if there exist an algebra morphism $\l
: \mf {B}\ra H\ot \mathfrak{B}$ and an invertible element $\Phi
_{\l }\in H\ot H\ot \mathfrak{B}$ such that:
\begin{eqnarray}
&&(id\ot \l )(\l (\mf {b}))\Phi _{\l }=\Phi _{\l
}(\Delta \ot id)(\l (\mf {b})),
\mbox{${\;\;\;}$$\forall $ $\mf {b}\in \mathfrak{B}$,}
\label{lca1}\\[1mm]%
&&(1_H\ot \Phi _{\l })(id\ot \Delta \ot id)(\Phi _{\l })(\Phi \ot
1_{\mf {B}})= (id\ot id\ot \l )(\Phi _{\l })(\Delta \ot id\ot
id)(\Phi _{\l }),\label{lca2}\\[1mm]%
&&(\va \ot id)\circ \l =id ,\label{lca3}\\[1mm]%
&&(id\ot \va \ot id)(\Phi _{\l })=(\va \ot id\ot id)(\Phi _{\l }
)=1_H\ot 1_{\mathfrak{B}}.\label{lca4}
\end{eqnarray}
Finally, by an $H$-bicomodule algebra $\mb {A}$ 
we mean a quintuple $(\l, \r , \Phi _{\l }, \Phi _{\r }, \Phi
_{\l , \r })$, where $\l $ and $\r $ are left and right
$H$-coactions on $\mb {A}$, respectively, and where $\Phi _{\l
}\in H\ot H\ot \mb {A}$, $\Phi _{\r }\in \mb {A}\ot H\ot H$ and
$\Phi _{\l , \r }\in H\ot \mb {A}\ot H$ are invertible elements,
such that 
$(\mb {A}, \l , \Phi _{\l })$ is a left $H$-comodule algebra, 
$(\mb {A}, \r , \Phi _{\r })$ is a right $H$-comodule algebra and 
the following compatibility relations hold:
\begin{eqnarray}
&&\hspace{-1.8cm}\Phi _{\l , \r }(\l \ot id)(\r (u))=(id\ot \r )(\l 
(u))\Phi _{\l, \r }, \mbox{${\;\;}$$\forall $ $u\in \mb  
{A}$,}\label{bca1}\\[1mm]%
&&\hspace{-1.8cm}(1_H\ot \Phi _{\l , \r })(id\ot \l \ot id)(\Phi
_{\l , \r }) (\Phi _{\l }\ot 1_H)=(id\ot id\ot \r )(\Phi _{\l
})(\Delta \ot
id\ot id)(\Phi _{\l , \r }), \label{bca2}\\[1mm]%
&&\hspace{-1.8cm}(1_H\ot \Phi _{\r })(id\ot \r \ot id)(\Phi _{\l ,\r
})(\Phi _{\l , \r }\ot 1_H)= (id\ot id\ot \Delta )(\Phi _{\l , \r
})(\l \ot id\ot id) (\Phi _{\r }).\label{bca3}
\end{eqnarray}
\end{definition}

As pointed out in \cite{hn1}, if $\mb {A}$ is a bicomodule algebra
then, in addition, we have:
\begin{eqnarray}\label{bca4}
(id_H\ot id_{\mb {A}}\ot \va )(\Phi _{\l , \r })=1_H\ot 1_{\mb
{A}}, \mbox{${\;\;}$} (\va \ot id_{\mb {A}}\ot id_H)(\Phi _{\l ,
\r })= 1_{\mb {A}} \ot 1_H.
\end{eqnarray}

An example of a bicomodule algebra is $\mb {A}=H$, $\l =\r  
=\Delta $ and $\Phi _{\l }=\Phi _{\r }= \Phi _{\l , \r }=\Phi $.

If $(B, \lambda , \rho , \Phi _{\l }, \Phi _{\r }, \Phi _{\l , \r })$  and 
 $(B', \lambda ', \rho ', \Phi _{\l '}, \Phi _{\r '}, \Phi _{\l ', \r '})$ 
are $H$-bicomodule algebras, a morphism of $H$-bicomodule algebras 
$f:B\rightarrow B'$ is an algebra map such that $\rho '\circ f=(f\otimes id_H)\circ \rho $, 
$\lambda '\circ f=(id_H\otimes f)\circ \lambda $, $\Phi _{\rho '}=(f\ot id_H\ot id_H)(\Phi _{\rho })$, 
$\Phi _{\lambda '}=(id_H\ot id_H\ot f)(\Phi _{\lambda })$ and 
$\Phi _{\l ', \r '}=(id_H\otimes f\otimes id_H)(\Phi _{\l , \r })$. 

Let us denote by $_H{\cal M}_H$ the category of $H$-bimodules; it is 
also a monoidal category, the 
associativity constraints being given by ${\bf a^{'}}_{U, V, W}:
(U\ot V)\ot W\ra U\ot (V\ot W)$,  
${\bf a^{'}}_{U, V, W}((u\ot v)\ot w)= \Phi \cd (u\ot (v\ot w))\cd
\Phi ^{-1}$,  
for $U, V, W\in {}_H{\cal M}_H$ and $u\in U$, $v\in V$, $w\in 
W$. Therefore, we can define coalgebras in the category of
$H$-bimodules; in particular, the axioms for $H$ ensure that $H$ is a (coassociative) 
coalgebra in $_H{\cal M}_H$.

We recall from \cite{majid} the definition of (left) Yetter-Drinfeld modules 
over a quasi-bialgebra $H$. 
\begin{definition}
A $k$-linear space $M$ is called a left Yetter-Drinfeld module over $H$ if 
$M$ is a left $H$-module (with action denoted by $h\ot m\mapsto h\cdot m$) 
and $H$ coacts on $M$ to the left (the coaction is denoted by 
$\lambda _M:M\rightarrow H\ot M$, $\lambda _M(m)=m^{(-1)}\ot m^{(0)}$) such 
that:
\begin{eqnarray}
&&X^1m^{(-1)}\ot (X^2\cdot m^{(0)})^{(-1)}X^3\ot (X^2\cdot m^{(0)})^{(0)}
\nonumber \\
&&\;\;\;\;\;\;\;\;\;
=X^1((Y^1\cdot m)^{(-1)})_1Y^2\ot X^2((Y^1\cdot m)^{(-1)})_2Y^3\ot  
X^3\cdot (Y^1\cdot m)^{(0)}, \label{yd1} \\
&&\varepsilon (m^{(-1)})m^{(0)}=m, \label{yd2} \\
&&h_1m^{(-1)}\ot h_2\cdot m^{(0)}=(h_1\cdot m)^{(-1)}h_2\ot 
(h_1\cdot m)^{(0)}, \label{yd3}
\end{eqnarray}
for all $m\in M$ and $h\in H$. 
The category $_H^H{\cal YD}$ consists of such objects, the morphisms in the 
category being the $H$-linear maps intertwining the $H$-coactions. 
\end{definition}

The category $_H^H{\cal YD}$ is (pre) braided monoidal; explicitly, 
if $(M, \lambda _M)$ and $(N, \lambda _N)$ are objects in $_H^H{\cal YD}$, 
then $(M\ot N, \lambda _{M\ot N})$ is also object in 
$_H^H{\cal YD}$, where $M\ot N$ is a left $H$-module with action $h\cdot 
(m\ot n)=h_1\cdot m\ot h_2\cdot n$, and the coaction $\lambda _{M\ot N}$ is 
given by 
\begin{eqnarray*}
&&\lambda _{M\ot N}(m\ot n)=X^1(x^1Y^1\cdot m)^{(-1)}x^2(Y^2\cdot n)^{(-1)}
Y^3\ot X^2\cdot (x^1Y^1\cdot m)^{(0)}\\ 
&&\;\;\;\;\;\;\;\;\;\;\;\;\;\;\;\;\;\;\;\;\;\;\;\;\;\;\;\;\;\;\;\;
\ot X^3x^3\cdot (Y^2\cdot n)^{(0)}.
\end{eqnarray*}
The associativity constraints are the same as in $_H{\cal M}$, and the 
(pre) braiding is given by 
\begin{eqnarray*}
&&c_{M, N}:M\ot N\rightarrow N\ot M, \;\;c_{M, N}(m\ot n)=m^{(-1)}\cdot n\ot 
m^{(0)}.
\end{eqnarray*}

Since $_H^H{\cal YD}$ is a monoidal category, we can speak about algebras 
in $_H^H{\cal YD}$. Namely, if $A$ is an object in $_H^H{\cal YD}$, then 
$A$ is an algebra in $_H^H{\cal YD}$ if and only if 
$A$ is a left $H$-module algebra and 
$A$ is a left quasi-comodule algebra, that is its unit and 
multiplication intertwine the $H$-coaction $\lambda _A$, namely 
(for all $a, a'\in A$):  
\begin{eqnarray}
&&\lambda _A(1_A)=1_H\ot 1_A, \label{unitate} \\
&&\lambda _A(aa')=X^1(x^1Y^1\cdot a)^{(-1)}x^2(Y^2\cdot a')^{(-1)}Y^3 
\nonumber \\
&&\;\;\;\;\;\;\;\;\;\;\;\;\;\;\;\;\;\;\;\;
\ot [X^2\cdot (x^1Y^1\cdot a)^{(0)}][X^3x^3\cdot (Y^2\cdot a')^{(0)}].
\label{multi}
\end{eqnarray}

We recall the following result from \cite{ap}:
\begin{proposition}\label{YD}
Let $H$ be a quasi-bialgebra and $A$ an algebra in $_H^H{\cal YD}$, with coaction denoted by 
$A\rightarrow H\ot A$, $a\mapsto a^{(-1)}\ot a^{(0)}$. 
Then 
$(A\# H, \lambda , \rho , \Phi _{\lambda }, \Phi _{\rho }, 
\Phi _{\lambda , \rho })$ is an $H$-bicomodule algebra, 
with structures:
\begin{eqnarray*}
&&\lambda :A\# H\rightarrow H\ot (A\# H), \;\;\;\;
\lambda (a\# h)=T^1(t^1\cdot a)^{(-1)}t^2h_1\ot (T^2\cdot 
(t^1\cdot a)^{(0)}\# T^3t^3h_2),\\
&&\rho :A\# H\rightarrow (A\# H)\otimes H, \;\;\;\;
\rho (a\# h)=(x^1\cdot a\# x^2h_1)\otimes x^3h_2, \\
&&\Phi _{\lambda }=X^1\ot X^2\ot (1_A\# X^3)\in H\ot H\ot (A\# H), \\
&&\Phi _{\rho}=(1_A\# X^1)\otimes X^2\otimes X^3\in (A\# H)\otimes H
\otimes H, \\
&&\Phi _{\lambda , \rho }=X^1\ot (1_A\# X^2)\ot X^3\in H\ot (A\# H)\ot H.
\end{eqnarray*}
\end{proposition}

Moreover, one can easily see that in the hypotheses of Proposition \ref{YD}, the map 
$H\rightarrow A\# H$, $h\mapsto 1_A\# h$, is a morphism of $H$-bicomodule algebras.

We prove now a partial converse of Proposition \ref{YD}.
\begin{proposition}\label{converseYD}
Let $H$ be a quasi-bialgebra and $A$ an object in $_H^H{\cal YD}$, with action and coaction 
denoted by $H\ot A\rightarrow A$, $h\ot a \mapsto h\cdot a$ and $\lambda _A:A\rightarrow 
H\ot A$, $a\mapsto a^{(-1)}\ot a^{(0)}$. Assume that $A$ is also a left $H$-module 
algebra (with respect to the action $h\ot a\mapsto h\cdot a$). Assume that the map 
\begin{eqnarray*}
&&\lambda :A\# H\rightarrow H\ot (A\# H), \;\;\;\;
\lambda (a\# h)=T^1(t^1\cdot a)^{(-1)}t^2h_1\ot (T^2\cdot 
(t^1\cdot a)^{(0)}\# T^3t^3h_2)
\end{eqnarray*}
is an algebra map. Then $A$ is an algebra in $_H^H{\cal YD}$.
\end{proposition}
\begin{proof}
The fact that $\lambda $ is unital implies immediately that $\lambda _A$ is unital, 
so the only thing left to prove is the relation (\ref{multi}) for $A$. Let $a, a'\in A$. 
Since $\lambda $ is multiplicative, we have 
$$\lambda ((Z^1\cdot a\# 1)(Z^2\cdot a'\# Z^3))=\lambda (Z^1\cdot a\# 1)
\lambda (Z^2\cdot a'\# Z^3)$$
We compute the left and right hand sides of this equality:
\begin{eqnarray*}
\lambda ((Z^1\cdot a\# 1)(Z^2\cdot a'\# Z^3))
&=&\lambda ((z^1Z^1\cdot a)(z^2Z^2\cdot a'))\# z^3Z^3)\\
&=&\lambda (aa'\# 1)\\
&=&T^1(t^1\cdot aa')^{(-1)}t^2\ot (T^2\cdot 
(t^1\cdot aa')^{(0)}\# T^3t^3),
\end{eqnarray*}
\begin{eqnarray*}
\lambda (Z^1\cdot a\# 1)\lambda (Z^2\cdot a'\# Z^3)&=&
[T^1(t^1Z^1\cdot a)^{(-1)}t^2\ot (T^2\cdot 
(t^1Z^1\cdot a)^{(0)}\# T^3t^3)]\\
&&[Y^1(y^1Z^2\cdot a')^{(-1)}y^2Z^3_1\ot (Y^2\cdot 
(y^1Z^2\cdot a')^{(0)}\# Y^3y^3Z^3_2)]\\
&=&T^1(t^1Z^1\cdot a)^{(-1)}t^2Y^1(y^1Z^2\cdot a')^{(-1)}y^2Z^3_1\ot \\
&&[x^1T^2\cdot (t^1Z^1\cdot a)^{(0)}][x^2T^3_1t^3_1Y^2\cdot 
(y^1Z^2\cdot a')^{(0)}]\#x^3T^3_2t^3_2Y^3y^3Z^3_2. 
\end{eqnarray*}
Now we apply $\varepsilon $ on the last position in both terms. We obtain: 
\begin{eqnarray*}
(id \ot \varepsilon )(\lambda ((Z^1\cdot a\# 1)(Z^2\cdot a'\# Z^3)))&=&
(aa')^{(-1)}\ot (aa')^{(0)}=\lambda _A(aa'), \\
(id \ot \varepsilon )(\lambda (Z^1\cdot a\# 1)\lambda (Z^2\cdot a'\# Z^3))&=&
T^1(t^1Z^1\cdot a)^{(-1)}t^2(Z^2\cdot a')^{(-1)}Z^3\ot \\ 
&&\;\;\;\;\;\;\;\;\;\;\;\;\;\;\;\;\;\;\;\
 [T^2\cdot (t^1Z^1\cdot a)^{(0)}][T^3t^3\cdot (Z^2\cdot a')^{(0)}].
\end{eqnarray*}
The equality of these two terms is exactly the desired relation (\ref{multi}). 
\end{proof}

Let $H$ be a quasi-bialgebra and $M$ an $H$-bimodule together with two 
$H$-bimodule maps $\rho :M\rightarrow M\ot H$, $\lambda :M\rightarrow H\ot M$, 
with notation 
$\rho (m)=m_{(0)}\ot m_{(1)}$ and $\lambda (m)=m_{<-1>}\ot m_{<0>}$, 
for $m\in M$ (called respectively a right and a left 
$H$-coaction on $M$).  
Then $(M, \lambda, \rho )$ is called a two-sided two-cosided  quasi-Hopf $H$-bimodule if 
$M$ is an $H$-bicomodule in the monoidal category $_H{\cal M}_H$, that is if the 
following conditions are satisfied, for all $m\in M$:
\begin{eqnarray}
&&(id_M\ot \varepsilon )\circ \rho =id_M, \label{qb1}\\
&&\Phi \cdot (\rho \ot id_H)(\rho (m))=(id_M\ot \Delta )(\rho (m))\cdot 
\Phi ,  \label{qb2}\\
&&(\varepsilon \ot id_M)\circ \lambda =id_M, \label{qb3}\\
&&(id_H\ot \lambda )(\lambda (m))\cdot \Phi =
\Phi \cdot (\Delta \ot id_M)(\lambda (m)), \label{qb4}\\
&&\Phi \cdot (\lambda \ot id_H)(\rho (m))=
(id_H\ot \rho )(\lambda (m))\cdot \Phi . \label{qb5}
\end{eqnarray}
The category of two-sided two-cosided  quasi-Hopf $H$-bimodules will be denoted by 
$^H_H{\cal M}_H^H$ (the morphisms in the category are the $H$-bimodule maps 
intertwining the $H$-coactions), cf. \cite{sch}. 

Let now $H$ be a quasi-Hopf algebra and $M$ an object in $^H_H{\cal M}_H^H$, 
with notation as above. 
Then in particular $M$ is also an object in the category $_H{\cal M}_H^H$ of quasi-Hopf 
$H$-bimodules introduced in \cite{hn}. So, following \cite{hn}, we can define the map 
$E:M\rightarrow M$ by the formula
\begin{eqnarray}
E(m)=q^1\cdot m_{(0)}\cdot \beta S(q^2m_{(1)}), \;\;\;\;\;\forall \;\;m\in M,
\end{eqnarray}
where $q_R=q^1\ot q^2=X^1\ot S^{-1}(\alpha X^3)X^2$.  
Also, for $h\in H$ and $m\in M$, define 
\begin{eqnarray}
&&h\triangleright m=E(h\cdot m). \label{act}
\end{eqnarray}
Some properties of $E$ and $\triangleright $ are collected in \cite{hn}, 
Proposition 3.4, for instance (for $h, h'\in H$ and $m\in M$): 
$E^2=E$; $E(m\cdot h)=E(m)\varepsilon (h)$; $h\triangleright E(m)=
E(h\cdot m)\equiv h\triangleright m$; $(hh')\triangleright m=h\triangleright  
(h'\triangleright m)$; $h\cdot E(m)=(h_1\triangleright E(m))\cdot h_2$; 
$E(m_{(0)})\cdot m_{(1)}=m$; $E(E(m)_{(0)})\ot E(m)_{(1)}=E(m)\ot 1$. 
Because of these properties, the following notions of {\it coinvariants} all 
coincide: 
\begin{eqnarray*}
&&M^{co (H)}=E(M)=\{n\in M/E(n)=n\}=\{n\in M/E(n_{(0)})\ot n_{(1)}=E(n)
\ot 1\}.
\end{eqnarray*}
From the above properties it follows that $(M^{co(H)}, \triangleright )$ is 
a left $H$-module. 

In \cite{sch}, Schauenburg proved a structure theorem for objects in $^H_H{\cal M}_H^H$, 
that can be reformulated as follows:
\begin{theorem}(\cite{sch}) \label{struct4corners}
Let $H$ be a quasi-Hopf algebra. \\
(i) Let $V\in \;$$^H_H{\cal YD}$, with $H$-action denoted by $\triangleright $ and $H$-coaction 
denoted by $V\rightarrow H\ot V$, $v\mapsto v^{(-1)}\ot v^{(0)}$. Then $V\ot H$ becomes an 
object in $^H_H{\cal M}_H^H$ with structures: 
\begin{eqnarray*}
&&a\cdot (v\ot h)\cdot b=(a_1\triangleright v)\ot a_2hb, \\
&&\lambda _{V\ot H}(v\ot h)=X^1(x^1\triangleright v)^{(-1)}x^2h_1\ot 
(X^2\triangleright (x^1\triangleright v)^{(0)}\ot X^3x^3h_2), \\
&&\rho _{V\ot H}(v\ot h)=(x^1\triangleright v\ot x^2h_1)\ot x^3h_2, 
\end{eqnarray*}
for all $a, b, h\in H$ and $v\in V$.\\
(ii) Let $M\in \;$$^H_H{\cal M}_H^H$. Consider $V=M^{co(H)}$ as a left $H$-module 
with action $\triangleright $ as in (\ref{act}) and define the map 
$V\rightarrow H\ot V$, $v\mapsto v_{<-1>}\ot E(v_{<0>})$, 
where we denoted by $M\rightarrow H\ot M$, $m\mapsto m_{<-1>}\ot m_{<0>}$ the left 
$H$-coaction on $M$. Then with these structures $V$ is an object in $^H_H{\cal YD}$, and 
if we regard $V\ot H\in \;$$^H_H{\cal M}_H^H$ as in (i), the map 
$\nu :V\ot H\rightarrow M$, $\nu (v\ot h)=v\cdot h$ is an 
isomorphism in $^H_H{\cal M}_H^H$.
\end{theorem}

For the sequel of this section, we fix a quasi-Hopf algebra $H$ and an $H$-bicomodule 
algebra $B$, with structure maps $\lambda _B$, $\rho _B$ and associators 
$\Phi _{\l _B}$, $\Phi _{\r _B}$, $\Phi _{\l _B, \r _B}$, with notation 
$\rho _B(b)=b_{(0)}\ot 
b_{(1)}\in B\ot H$ and $\lambda _B(b)=b_{<-1>}\ot b_{<0>}\in H\ot B$,  
such that there exists $v:H\rightarrow B$ a morphism of 
$H$-bicomodule algebras (in particular, this implies $\rho _B(v(h))=
v(h_1)\ot h_2$, $\lambda _B(v(h))=h_1\ot v(h_2)$, 
for all $h\in H$, and  
$\Phi _{\rho _B}=v(X^1)\ot X^2\ot X^3$, $\Phi _{\lambda _B}=X^1\ot X^2\ot v(X^3)$ 
and $\Phi _{\l _B, \r _B}=X^1\ot v(X^2)\ot X^3$). 
\begin{lemma}\label{4colturi}
$(B, \lambda _B, \rho _B)$ becomes an object in $^H_H{\cal M}_H^H$. 
\end{lemma} 
\begin{proof}
Obviously,  $B$ becomes an $H$-bimodule via $v$ (i.e. $h\cdot b\cdot h'=
v(h)bv(h')$ for all $h, h'\in H$ and $b\in B$). In \cite{pvo}, Lemma 2.3, was proved that 
$\rho _B:B\rightarrow B\ot H$ is an $H$-bimodule map and that the conditions (\ref{qb1}) and 
(\ref{qb2}) for $B$ are satisfied. Similarly one can prove that $\lambda _B:B\rightarrow H\ot B$ is 
an $H$-bimodule map and that the conditions (\ref{qb3}) and 
(\ref{qb4}) for $B$ are satisfied. Finally, the condition (\ref{qb5}) is also satisfied, because 
it reduces to the condition (\ref{bca1}) from the definition of an $H$-bicomodule algebra, 
due to the fact that $\Phi _{\l _B, \r _B}=X^1\ot v(X^2)\ot X^3$.
\end{proof}

We can prove now the structure theorem for quasi-Hopf bicomodule algebras. 
\begin{theorem}
Let $H$ be a quasi-Hopf algebra, $B$ an $H$-bicomodule algebra and $v:H\rightarrow B$ 
a morphism of $H$-bicomodule algebras. Regard $B\in \;$$^H_H{\cal M}_H^H$ as in 
Lemma \ref{4colturi} and define $A=B^{co(H)}$. Then $A$ is an algebra in $^H_H{\cal YD}$ and, 
if we regard $A\# H$ as an $H$-bicomodule algebra as in Proposition \ref{YD}, then 
the map $\Psi :A\# H\rightarrow B$, $\Psi (a\# h)=av(h)$, is an 
isomorphism of $H$-bicomodule algebras. 
\end{theorem}
\begin{proof}
Since $v$ is in particular a morphism of right $H$-comodule algebras, we know from \cite{pvo} 
that $A$ endowed with a certain multiplication and with the $H$-action given by (\ref{act}) 
becomes a left $H$-module algebra and the map $\Psi :A\# H\rightarrow B$ 
defined above is an isomorphism of 
right $H$-comodule algebras. It is very easy to see that $\Psi $ respects the left and two-sided associators, 
so the only things left to prove are that $A$ is an algebra in $^H_H{\cal YD}$ and that 
$\Psi $ intertwines the left $H$-coactions of $B$ and $A\# H$. By Theorem \ref{struct4corners}, 
we obtain that $A$ is an object in $^H_H{\cal YD}$ if we endow it with the coaction 
$a\mapsto a_{<-1>}\ot E(a_{<0>}):=a^{(-1)}\ot a^{(0)}$. Also, from Theorem 
\ref{struct4corners}, we have that the map $\Psi :A\ot  H\rightarrow B$ is a morphism 
in $^H_H{\cal M}_H^H$, in particular we have $\lambda _B\circ \Psi =
(id_H\ot \Psi )\circ \lambda _{A\# H}$, i.e. $\lambda _{A\# H}=(id_H\ot \Psi ^{-1})\circ 
\lambda _B\circ \Psi $, where $\lambda _{A\# H}(a\ot h)=X^1(x^1\triangleright a)^{(-1)}x^2h_1\ot 
(X^2\triangleright (x^1\triangleright a)^{(0)}\ot X^3x^3h_2)$. Since $\Psi $ and $\lambda _B$ 
are algebra maps, it follows that  $\lambda _{A\# H}$ is also an algebra map. We can now apply 
Proposition \ref{converseYD} to obtain that $A$ is an algebra in $^H_H{\cal YD}$. 
Finally, the fact that $\Psi $ intertwines the left $H$-coactions on $B$ and $A\# H$ 
follows from the fact that  $\lambda _B\circ \Psi =
(id_H\ot \Psi )\circ \lambda _{A\# H}$ and the fact that the $H$-coaction of the comodule 
algebra $A\# H$, as defined in Proposition \ref{YD}, is exactly the map 
$\lambda _{A\# H}$ defined above.
\end{proof}
\section{Weak Hopf bicomodule algebras}\label{sec2}
\setcounter{equation}{0}
${\;\;\;}$Following \cite{bns}, a weak Hopf algebra $H$ is a linear space such that $(H, \mu , 1)$ is an 
associative unital algebra, $(H, \Delta, \varepsilon )$ is a coassociative counital coalgebra and there exists 
a $k$-linear map $S:H\rightarrow H$ (called the antipode),  such that 
the following axioms hold:
\begin{eqnarray}
&&\Delta (hh')=\Delta (h)\Delta (h'),\;\;\;\forall \;h, h'\in H, \\
&&\Delta ^2(1)=(\Delta (1)\ot 1)(1\ot \Delta (1))=(1\ot \Delta (1))(\Delta (1)\ot 1), \label{delta21} \\
&&\varepsilon (xyz)=\varepsilon (xy_1)\varepsilon (y_2z)=\varepsilon (xy_2)\varepsilon 
(y_1z), \;\;\;\forall \;x, y, z\in H, \label{exyz}\\
&&h_1S(h_2)=\varepsilon (1_1h)1_2, \;\;\;\forall \;h\in H,\\
&&S(h_1)h_2=1_1\varepsilon (h1_2), \;\;\;\forall \;h\in H,\\
&&S(h_1)h_2S(h_3)=S(h), \;\;\;\forall \;h\in H.
\end{eqnarray}

 For such $H$, there exist two idempotent maps 
$\varepsilon _t, \varepsilon _s:H\rightarrow H$ defined by $\varepsilon _t(h)=\varepsilon (1_1h)1_2$, 
$\varepsilon _s(h)=1_1\varepsilon (h1_2)$, for all $h\in H$, called the target map and respectively the 
source map; their images, denoted by $H_t$ and respectively $H_s$, are called the target space and respectively the source space. 

For a weak Hopf algebra $H$, the following relations hold (see \cite{bns}, 
\cite{cwy} for proofs):
\begin{eqnarray}
&&h_1S(h_2)=\varepsilon _t(h), \;\;\;S(h_1)h_2=\varepsilon _s(h), \\
&&1_1\ot \varepsilon _t(1_2)=1_1\ot 1_2=\varepsilon _s(1_1)\ot 1_2, \\
&&\varepsilon _t(h\varepsilon _t(h'))=\varepsilon _t(hh'), \;\;\;
\varepsilon _s(\varepsilon _s(h)h')=\varepsilon _s(hh'), \\
&&\Delta (H_t)\subseteq H\ot H_t, \;\;\;\Delta (H_s)\subseteq H_s\ot H, \\
&&h_1\ot \varepsilon _t(h_2)=1_1h\ot 1_2, \;\;\;\varepsilon _s(h_1)\ot h_2=
1_1\ot h1_2, \\
&&h\varepsilon _t(h')=\varepsilon (h_1h')h_2, \;\;\;\varepsilon _s(h)h'=
h'_1\varepsilon (hh'_2), \label{cucu}\\
&&\varepsilon _t(\varepsilon _t(h)h')=\varepsilon _t(h)\varepsilon _t(h'), \;\;\;
\varepsilon _s(h\varepsilon _s(h'))=\varepsilon _s(h)\varepsilon _s(h'), \label{lala} \\
&&\varepsilon _t(h_1)h_2=h=h_1\varepsilon _s(h_2),  \label{est} \\
&&\Delta (1)=1_1\ot \varepsilon _t(1_2)=\varepsilon _s(1_1)\ot 1_2\in H_s\ot H_t, \label{delta1}\\
&&h_1\ot \varepsilon _s(h_2)=h1_1\ot S(1_2), \;\;\;
\varepsilon _t(h_1)\ot h_2=S(1_1)\ot 1_2h,  \label{titi}\\
&&\varepsilon (h\varepsilon _t(h'))=\varepsilon (hh')=
\varepsilon (\varepsilon _s(h)h'), \label{dudu}
\end{eqnarray}
for all $h, h'\in H$. 
Moreover, $H_t$ and $H_s$ are subalgebras of $H$ (containing $1$) and, for all $h\in H$, 
$y\in H_s$ and $z\in H_t$, the following relations hold:
\begin{eqnarray}
&&yz=zy, \label{comutst} \\
&&\Delta (y)=1_1\ot y1_2=1_1\ot 1_2y, \\
&&\Delta (z)=1_1z\ot 1_2=z1_1\ot 1_2, \label{deltaz} \\
&&y1_1\ot S(1_2)=1_1\ot S(1_2)y, \\
&&zS(1_1)\ot 1_2=S(1_1)\ot 1_2z, \\
&&h_1y\ot h_2=h_1\ot h_2S(y), \\
&&h_1\ot zh_2=S(z)h_1\ot h_2. \label{2.31b}
\end{eqnarray}

 Let $H$ be a weak Hopf algebra as above and $(A, \mu _A, 1_A)$ an associative 
unital algebra. Then $A$ is called a left $H$-module algebra (see for instance \cite{nik}) if 
$A$ is a left $H$-module with action denoted by $H\ot A\rightarrow A$, $h\ot a\mapsto h\cdot a$, 
satisfying the conditions: 
\begin{eqnarray}
&&h\cdot (ab)=(h_1\cdot a)(h_2\cdot b), \;\;\;
h\cdot 1_A=\varepsilon _t(h)\cdot 1_A, \label{modalg1}
\end{eqnarray}
for all $h\in H$ and $a, b\in A$. If this is the case, we can define the smash product 
$A\# H$, which, as a linear space, is the (relative) tensor product $A\ot _{H_t}H$, where $H$ 
is a left $H_t$-module via multiplication and $A$ is a right $H_t$-module as follows: 
$a\cdot z:=a(z\cdot 1_A)$, for all $a\in A$, $z\in H_t$. 
This linear space $A\# H$ becomes an associative algebra with unit $1_A\# 1_H$ and 
multiplication defined by 
$(a\# h)(a'\# h')=a(h_1\cdot a')\# h_2h'$, 
for all $a, a'\in A$ and $h, h'\in H$, where we denoted by $a\# h$ the class of $a\ot h$ in 
$A\ot _{H_t}H$. 
\begin{definition} (\cite{bohm})
Let $H$ be a weak Hopf algebra and $(A, \mu _A, 1_A)$ an associative 
unital algebra. \\
(i) $A$ is called a right $H$-comodule algebra if there is a linear map 
$\rho :A\rightarrow A\ot H$ such that:
\begin{eqnarray}
&&(id_A\ot \varepsilon )\circ \rho =id_A, \\
&&(\rho \ot id_H)\circ \rho =(id_A\ot \Delta )\circ \rho, \label{2.2a} \\
&&\rho (1_A)(a\ot 1_H)=((id_A\ot \varepsilon _t)\circ \rho )(a), \;\;\;\forall \;a\in A, \label{2.2b} \\
&&\rho (ab)=\rho (a)\rho (b),\;\;\;\forall \;a, b\in A.
\end{eqnarray}
(ii) $A$ is called a left $H$-comodule algebra if there is a linear map 
$\lambda :A\rightarrow H\ot A$ such that: 
\begin{eqnarray}
&&(\varepsilon \ot id_A)\circ \lambda =id_A, \\
&&(id_H\ot \lambda )\circ \lambda =(\Delta \ot id_A)\circ \lambda, \label{2.1a} \\
&&(1_H\ot a)\lambda (1_A)=((\varepsilon _s\ot id_A)\circ \lambda )(a), \;\;\;\forall \;a\in A, 
\label{2.1b} \\
&&\lambda (ab)=\lambda (a)\lambda (b), \;\;\;\forall \;a, b\in A. \label{2.1c}
\end{eqnarray}
(iii) A is called an $H$-bicomodule algebra if it is a right and left $H$-comodule algebra 
and the coactions $\rho $ and $\lambda $ satisfy the bicomodule condition 
$(\lambda \ot id_H)\circ \rho =(id_H\ot \rho )\circ \lambda $. If $A$, $B$ are two 
$H$-bicomodule algebras, a morphism 
of $H$-bicomodule algebras $f:A\rightarrow B$ is an algebra map intertwining the right and 
left coactions. 
\end{definition}

One can see that the condition (\ref{2.1b}) may be replaced by any of the following two 
equivalent conditions (that appear in \cite{nsw}, respectively \cite{nv}):
\begin{eqnarray}
&&(\Delta \ot id_A)(\lambda (1_A))=(1_H\ot \lambda (1_A))(\Delta (1_H)\ot 1_A), 
\label{NSW} \\
&&\lambda (1_A)=(\varepsilon _s\ot id_A)(\lambda (1_A)). \label{NV}
\end{eqnarray}

If $H$ is a weak Hopf algebra and $A$ is a left $H$-module algebra, then $A\# H$ 
becomes a right $H$-comodule algebra, with coaction $\rho :A\# H\rightarrow 
(A\# H)\ot H$, $\rho (a\# h)=(a\# h_1)\ot h_2$. 
\begin{definition} (\cite{wang}) 
Let $H$ be a weak Hopf algebra. A weak Hopf bimodule $M$ over $H$ is a linear space 
which is an $H$-bimodule and an $H$-bicomodule such that the two coactions are 
morphisms of $H$-bimodules. The category whose objects are weak Hopf bimodules and 
morphisms are linear maps intertwining the bimodule and bicomodule structures is denoted by 
$_H^H{\mathcal M}_H^H$.  
\end{definition}

Similarly, we can define the category $_H{\mathcal M}_H^H$. If $M$ is an object in this category, 
with $H$-module structures denoted by $\cdot $ and right $H$-comodule structure denoted by 
$\rho (m)=m_{(0)}\ot m_{(1)}$, define the map $E:M\rightarrow M$, 
$E(m)=m_{(0)}\cdot S(m_{(1)})$ and $M^{co (H)}:=\{m\in M/\rho (m)=
m_{(0)}\ot \varepsilon _t(m_{(1)})\}$. Then, by 
\cite{zhang}, \cite{yin}, $M^{co(H)}$ is a left $H$-module with action: 
\begin{eqnarray}
&&h\triangleright m:=E(h\cdot m), \;\;\;\forall \;h\in H, m\in M^{co(H)}. \label{tria}
\end{eqnarray}
\begin{definition} (\cite{cwy}) Let $H$ be a weak Hopf algebra. A (left-left) Yetter-Drinfeld module 
over $H$ is a linear space $M$ with a left $H$-module structure (denoted by $h\ot m\mapsto h\cdot m$) 
and a left $H$-comodule structure (denoted by $m\mapsto m_{(-1)}\ot m_{(0)}\in H\ot M$) 
such that the following conditions are satisfied, for all $h\in H$, $m\in M$:
\begin{eqnarray}
&&m_{(-1)}\ot m_{(0)}=1_1m_{(-1)}\ot 1_2\cdot m_{(0)}, \label{wyd1} \\
&&(h_1\cdot m)_{(-1)}h_2\ot (h_1\cdot m)_{(0)}=
h_1m_{(-1)}\ot h_2\cdot m_{(0)}. \label{wyd2}
\end{eqnarray}
We denote by $_H^H{\mathcal YD}$ the category whose objects are Yetter-Drinfeld modules and 
morphisms are $H$-linear $H$-colinear maps. 
\end{definition}

Exactly as in the Hopf case, condition (\ref{wyd2}) may be replaced by the equivalent condition 
\begin{eqnarray}
&&(h\cdot m)_{(-1)}\ot (h\cdot m)_{(0)}=h_1m_{(-1)}S(h_3)\ot h_2\cdot m_{(0)}. \label{wyd3}
\end{eqnarray}
\begin{theorem} \label{weakYD}
Assume that $H$ is a weak Hopf algebra and $A$ is a linear space such that 
$A$ is a left $H$-module algebra (with action $h\ot a\mapsto h\cdot a$), $A$ is a left $H$-comodule 
algebra (with coaction $\lambda :A\rightarrow H\ot A$, $\lambda (a):=a^{(-1)}\ot a^{(0)}$) and 
$(A, \cdot , \lambda )$ is a (left-left) Yetter-Drinfeld module. Then $A\# H$ is an $H$-bicomodule 
algebra, with coactions
\begin{eqnarray*}
&&\rho :A\# H\rightarrow 
(A\# H)\ot H, \;\;\;\rho (a\# h)=(a\# h_1)\ot h_2, \\
&&\lambda :A\# H\rightarrow H\ot (A\# H), \;\;\;\lambda (a\# h)=a^{(-1)}h_1\ot 
(a^{(0)}\# h_2),   
\end{eqnarray*}
and the linear map $j:H\rightarrow A\# H$, $j(h)=1_A\# h$, is a morphism of 
$H$-bicomodule algebras. 
\end{theorem}
\begin{proof}
Some of the conditions to be checked are trivial, we will prove only the nontrivial ones. \\
We begin by noting that, 
with a proof similar to the one in 
Remark 2.6 in \cite{nenciu}, and respectively as a consequence of (\ref{NV}), we have
the following relations: 
\begin{eqnarray}
&&y\cdot a=\varepsilon (a^{(-1)}y)a^{(0)}, \;\;\;\forall \;a\in A, y\in H_s,  \label{calanen} \\
&&1_A^{(-1)}\ot 1_A^{(0)}\in H_s\ot A. \label{consec}
\end{eqnarray}

We prove first that $\lambda $ is well-defined, that is 
$\lambda (a\# zh)=\lambda (a(z\cdot 1_A)\# h)$, for all $a\in A$, $h\in H$, $z\in H_t$. 
We compute:
\begin{eqnarray*}
\lambda (a\# zh)&=&a^{(-1)}z_1h_1\ot (a^{(0)}\# z_2h_2)\\
&\overset{(\ref{deltaz})}{=}&a^{(-1)}z1_1h_1\ot (a^{(0)}\# 1_2h_2)\\
&=&a^{(-1)}zh_1\ot (a^{(0)}\# h_2),
\end{eqnarray*}
\begin{eqnarray*}
\lambda (a(z\cdot 1_A)\# h)&=&[a(z\cdot 1_A)]^{(-1)}h_1\ot ([a(z\cdot 1_A)]^{(0)}\# h_2)\\
&=&a^{(-1)}(z\cdot 1_A)^{(-1)}h_1\ot (a^{(0)}(z\cdot 1_A)^{(0)}\# h_2)\\
&\overset{(\ref{wyd3})}{=}&a^{(-1)}z_11_A^{(-1)}S(z_3)h_1\ot 
(a^{(0)}(z_2\cdot 1_A^{(0)})\# h_2)\\
&\overset{(\ref{deltaz})}{=}&a^{(-1)}z1_11_A^{(-1)}S(1_3)h_1\ot 
(a^{(0)}(1_2\cdot 1_A^{(0)})\# h_2)\\
&\overset{(\ref{delta21})}{=}&a^{(-1)}z1_11_A^{(-1)}S(1_{2'})h_1\ot 
(a^{(0)}(1_{1'}1_2\cdot 1_A^{(0)})\# h_2)\\
&\overset{(\ref{wyd1})}{=}&a^{(-1)}z1_A^{(-1)}S(1_{2})h_1\ot 
(a^{(0)}(1_{1}\cdot 1_A^{(0)})\# h_2)\\
&\overset{(\ref{delta1}), \;(\ref{2.31b})}{=}&a^{(-1)}z1_A^{(-1)}h_1\ot 
(a^{(0)}(1_{1}\cdot 1_A^{(0)})\# 1_2h_2)\\
&\overset{(*)}{=}&a^{(-1)}z1_A^{(-1)}h_1\ot 
(a^{(0)}(1_{1}\cdot 1_A^{(0)})(1_2\cdot 1_A)\# h_2)\\
&\overset{(\ref{modalg1})}{=}&a^{(-1)}z1_A^{(-1)}h_1\ot 
(a^{(0)}1_A^{(0)}\# h_2)\\
&\overset{(\ref{consec}), \;(\ref{comutst})}{=}&a^{(-1)}1_A^{(-1)}zh_1\ot 
(a^{(0)}1_A^{(0)}\# h_2)\\
&=&a^{(-1)}zh_1\ot 
(a^{(0)}\# h_2), 
\end{eqnarray*}
where the equality $(*)$ follows by using the fact that $A\# H$ is the tensor product over $H_t$. 

We prove now the counitality condition for $\lambda $, i.e. 
$\varepsilon (a^{(-1)}h_1)a^{(0)}\# h_2=a\# h$, 
for all $a\in A$, $h\in H$. We compute:
\begin{eqnarray*}
\varepsilon (a^{(-1)}h_1)a^{(0)}\# h_2
&\overset{(\ref{exyz})}{=}&\varepsilon (a^{(-1)}1_1)\varepsilon (1_2h_1)a^{(0)}\# h_2\\
&\overset{(\ref{delta1}), \;(\ref{calanen})}{=}&\varepsilon (1_2h_1)(1_1\cdot a)\# h_2\\
&=&\varepsilon (1_21_{1'}h_1)(1_1\cdot a)\# 1_{2'}h_2\\
&\overset{(\ref{delta21})}{=}&\varepsilon (1_{1'}1_2h_1)(1_1\cdot a)\# 1_{2'}h_2\\
&\overset{(\ref{cucu})}{=}&1_1\cdot a\# \varepsilon _t(1_2h_1)h_2\\
&\overset{(\ref{delta1})}{=}&1_1\cdot a\# \varepsilon _t(\varepsilon _t(1_2)h_1)h_2\\
&\overset{(\ref{lala})}{=}&1_1\cdot a\# \varepsilon _t(1_2)\varepsilon _t(h_1)h_2\\
&\overset{(\ref{est})}{=}&1_1\cdot a\# \varepsilon _t(1_2)h\\
&\overset{(\ref{delta1})}{=}&1_1\cdot a\# 1_2h\\
&\overset{(*)}{=}&(1_1\cdot a)(1_2\cdot 1_A)\# h\\
&\overset{(\ref{modalg1})}{=}&1\cdot (a1_A)\# h
=a\# h, 
\end{eqnarray*}
where again for proving the equality $(*)$ we used the fact that the tensor product is over $H_t$.

We prove now the condition (\ref{2.1b}) in the definition of a 
left $H$-comodule algebra. As we have seen, this is equivalent to the condition 
(\ref{NV}), so it is enough to prove (\ref{NV}) for our $\lambda $, namely 
$\lambda (1_A\# 1_H)=(\varepsilon _s\ot id_A\ot id_H)(\lambda (1_A\# 1_H))$.  We compute:
\begin{eqnarray*}
(\varepsilon _s\ot id_A\ot id_H)(\lambda (1_A\# 1_H))&=&
(\varepsilon _s\ot id_A\ot id_H)(1_A^{(-1)}1_1\ot (1_A^{(0)}\# 1_2))\\
&=&\varepsilon _s(1_A^{(-1)}1_1)\ot (1_A^{(0)}\# 1_2)\\
&\overset{(\ref{delta1})}{=}&\varepsilon _s(1_A^{(-1)}\varepsilon _s(1_1))\ot (1_A^{(0)}\# 1_2)\\
&\overset{(\ref{lala})}{=}&\varepsilon _s(1_A^{(-1)})\varepsilon _s(1_1)\ot (1_A^{(0)}\# 1_2)\\
&\overset{(\ref{delta1})}{=}&\varepsilon _s(1_A^{(-1)})1_1\ot (1_A^{(0)}\# 1_2)\\
&\overset{(\ref{NV})}{=}&1_A^{(-1)}1_1\ot (1_A^{(0)}\# 1_2)
=\lambda (1_A\# 1_H),  \;\;\;q.e.d.
\end{eqnarray*}

The only nontrivial thing left to prove is that the map $j$ intertwines the left coactions, that is 
$1_A^{(-1)}h_1\ot 1_A^{(0)}\# h_2=h_1\ot 1_A\# h_2$, for all $h\in H$. We compute:
\begin{eqnarray*}
1_A^{(-1)}h_1\ot 1_A^{(0)}\# h_2&\overset{(\ref{NV})}{=}&
\varepsilon _s(1_A^{(-1)})h_1\ot 1_A^{(0)}\# h_2\\
&\overset{(\ref{cucu})}{=}&h_1\ot \varepsilon (1_A^{(-1)}h_2)1_A^{(0)}\# h_3\\
&\overset{(\ref{dudu})}{=}&h_1\ot \varepsilon (1_A^{(-1)}\varepsilon _t(h_2))1_A^{(0)}\# h_3\\
&\overset{(\ref{titi})}{=}&h_1\ot \varepsilon (1_A^{(-1)}S(1_1))1_A^{(0)}\# 1_2h_2\\
&\overset{(*)}{=}&h_1\ot [\varepsilon (1_A^{(-1)}S(1_1))1_A^{(0)}][1_2\cdot 1_A]\# h_2\\
&\overset{(\ref{titi}), \;(\ref{dudu})}{=}&
h_1\ot [\varepsilon (1_A^{(-1)}1_1)1_A^{(0)}][1_2\cdot 1_A]\# h_2\\
&\overset{(\ref{delta1}), \;(\ref{calanen})}{=}&
h_1\ot [1_1\cdot 1_A][1_2\cdot 1_A]\# h_2\\
&\overset{(\ref{modalg1})}{=}&h_1\ot 1_A\# h_2, 
\end{eqnarray*}
where again for proving the equality $(*)$ we used the fact that the tensor product is over $H_t$.
\end{proof}

Let $H$ be a weak Hopf algebra with bijective antipode. It was proved in \cite{wangzhu} that 
there exists an equivalence of categories between $_H^H{\mathcal M}_H^H$ and 
the category of right-right Yetter-Drinfeld modules over $H$. We will need the 
left-handed analogue of this result, whose proof is analogous to the one in \cite{wangzhu} 
and is left to the reader: 
\begin{proposition} \label{weakstruct4corners}
Let $H$ be a weak Hopf algebra with bijective antipode. \\
(i) Let $V\in \;_H^H{\mathcal YD}$, with $H$-action denoted by $\triangleright $ and 
$H$-coaction $V\rightarrow H\ot V$, $v\mapsto v^{(-1)}\ot v^{(0)}$. Then $V\ot _{H_t}H$ 
becomes an object in $_H^H{\mathcal M}_H^H$ with structures:
\begin{eqnarray*}
&&a\cdot (v\ot h)=a_1\triangleright v\ot a_2h, \;\;\;
(v\ot h)\cdot b=v\ot hb, \\
&&\lambda _{V\ot _{H_t}H}(v\ot h)=v^{(-1)}h_1\ot (v^{(0)}\ot h_2), \\
&&\rho _{V\ot _{H_t}H}(v\ot h)=(v\ot h_1)\ot h_2, 
\end{eqnarray*}
for all $a, b, h\in H$ and $v\in V,$ where $V$ is regarded as a right $H_t$-module by the formula 
$v\cdot z=S(z)\triangleright v$, for all $v\in V$ and $z\in H_t$. \\
(ii) Let $M\in \;_H^H{\mathcal M}_H^H$. Consider $V=M^{co(H)}$ as a left $H$-module 
with action $\triangleright $ as in (\ref{tria}) and define the map $V\rightarrow H\ot V$, 
$v\mapsto v_{<-1>}\ot v_{<0>}$, where we denoted by $M\rightarrow H\ot M$, 
$m\mapsto m_{<-1>}\ot m_{<0>}$ the left $H$-coaction on $M$. Then with these structures $V$ is 
an object in $_H^H{\mathcal YD}$, and if we regard $V\ot _{H_t}H\in \;_H^H{\mathcal M}_H^H$ 
as in (i), the map $\nu :V\ot _{H_t}H\rightarrow M$, $\nu (v\ot h)=v\cdot h$ is an isomorphism 
in $_H^H{\mathcal M}_H^H$. 
\end{proposition}

We fix now a weak Hopf algebra $H$ with bijective antipode and an 
$H$-bicomodule algebra $B$, with coactions $\lambda _B$, $\rho _B$ with notation 
$\rho _B(b)=b_{(0)}\ot b_{(1)}\in B\ot H$ and 
$\lambda _B(b)=b_{<-1>}\ot b_{<0>}\in H\ot B$, such that there exists $v:H\rightarrow B$ a 
morphism of $H$-bicomodule algebras; in particular, this implies $\rho _B(v(h))=v(h_1)\ot h_2$ 
and $\lambda _B(v(h))=h_1\ot v(h_2)$, for all $h\in H$. 
\begin{lemma}
$B$ becomes an object in $_H^H{\mathcal M}_H^H$, with coactions $\lambda _B$ and $\rho _B$ and 
actions $h\cdot b=v(h)b$ and 
$b\cdot h=bv(h)$, for all $h\in H$ and $b\in B$. 
\end{lemma}
\begin{proof}
A straightforward computation. 
\end{proof}

Since $B$ is an object in $_H^H{\mathcal M}_H^H$, we can consider the coinvariants 
$A:=B^{co(H)}$. By \cite{zhang}, \cite{yin} we know that $(A, \triangleright , 1_B)$ 
is a left $H$-module algebra (where the action $\triangleright $ is defined as above by 
$h\triangleright a=E(h\cdot a)=E(v(h)a)$, for all $h\in H$, $a\in A$, and the 
multiplication of $A$ is the restriction to $A$ of the multiplication of $B$) and the map 
$\phi :A\# H\rightarrow B$, $\phi (a\# h)=av(h)$ is an isomorphism of 
right $H$-comodule algebras. By Proposition \ref{weakstruct4corners} we know that $A$ is an 
object in $_H^H{\mathcal YD}$. 
\begin{theorem}
With notation as above, $A$ is also a left $H$-comodule algebra, and if we regard $A\# H$ 
as an $H$-bicomodule algebra as in Theorem \ref{weakYD} then the map 
$\phi :A\# H\rightarrow B$, $\phi (a\# h)=av(h)$ is an isomorphism of 
$H$-bicomodule algebras.
\end{theorem}
\begin{proof}
$A$ is a left $H$-comodule algebra because its multiplication and left $H$-comodule structure are 
the restrictions of the ones of $B$, the unit of $A$ is the same as the unit of $B$ and $B$ is a 
left $H$-comodule algebra. So the only thing left to prove is that the map $\phi $ intertwines 
the left $H$-coactions on $A\# H$ and $B$, and this follows by a straightforward 
computation using the fact that $\lambda _B(v(h))=h_1\ot v(h_2)$, for all $h\in H$.  
\end{proof}
\section{Braided bicomodule algebras}\label{secbraided}
\setcounter{equation}{0}
${\;\;\;}$
In this section, $(\C,\otimes,I,a,l,r,\phi)$ will denote a braided monoidal category. General references on 
(braided) category theory are \cite{joyalstreet}, \cite{maclane}. In view of Mac Lane's coherence theorem, 
we can (and will) assume, without loss of generality, that $\C$ is a strict monoidal category, i.e. 
the associativity and unity constraints are given by identity. Furthermore, we will freely make use of 
graphical calculus. First, the braiding and its inverse are denoted by:
$$
\phi_{M,N} \,
=
\gbeg{2}{3}
\got{1}{M}\got{1}{N}\gnl
\gbr\gnl
\gob{1}{N}\gob{1}{M}
\gend
\quad
\text{and}
\quad
\phi\inv_{M,N} \,
=
\gbeg{2}{3}
\got{1}{N}\got{1}{M}\gnl
\gibr\gnl
\gob{1}{M}\gob{1}{N}
\gend
$$

For an algebra $A$ and a coalgebra $C$ in $\C$, we denote the multiplication and unit of $A$, 
and the comultiplication and counit of $C$ by:
$$
\nabla_A \,
=
\gbeg{2}{3}
\got{1}{A}\got{1}{A}\gnl
\gmu\gnl
\gob{2}{A}
\gend
\quad , \quad
\eta_A \,
=
\gbeg{1}{3}
\gvac{1}\gnl
\gu{1}\gnl
\gob{1}{A}
\gend
\quad
\text{and}
\quad
\Delta_C \,
=
\gbeg{2}{3}
\got{2}{C}\gnl
\gcmu\gnl
\gob{1}{C}\gob{1}{C}
\gend 
\quad , \quad
\varepsilon_C \,
=
\gbeg{1}{3}
\got{1}{C}\gnl
\gcu{1}\gnl
\gnl
\gend
$$

A bialgebra $(H,\nabla,\eta,\Delta,\varepsilon)$ in a braided monoidal category $\C$ is simultaneously an 
algebra $(H,\nabla,\eta)$ and a coalgebra $(H,\Delta,\varepsilon)$ such that $\Delta$ and $\varepsilon$ are 
algebra morphisms (e.g. \cite{majidbraided}). Such an $H$ is also called a braided bialgebra. 
Graphically, $H$ satisfies:
\begin{align}\label{eqbialgebra}
\gbeg{2}{4}
\got{1}{H}\got{1}{H}\gnl
\gmu\gnl
\gcmu\gnl
\gob{1}{H}\gob{1}{H}
\gend
=
\gbeg{4}{5}
\got{2}{H}\got{2}{H}\gnl
\gcmu\gcmu\gnl
\gcl{1}\gbr\gcl{1}\gnl
\gmu\gmu\gnl
\gob{2}{H}\gob{2}{H}
\gend
\quad
\text{and}
\quad
\gbeg{2}{3}
\got{1}{H}\got{1}{H}\gnl
\gmu\gnl
\gcu{2}
\gend
=
\gbeg{2}{3}
\got{1}{H}\got{1}{H}\gnl
\gcu{1}
\gcu{1}
\gend
\end{align}

A braided Hopf algebra $H$ is a braided bialgebra together with a morphism $S : H \rightarrow H$ in $\C$, 
called the antipode, satisfying:
\begin{align}
\nabla \circ (S \otimes H) \circ \Delta = \nabla \circ (H \otimes S) \circ \Delta = \eta \circ \varepsilon \label{eqantipodedef}
\end{align}

We denote the antipode $S$ (and its inverse, if it exists) by:
$$
S \,
=
\gbeg{1}{5}
\got{1}{H}\gnl
\gcl{1}\gnl
\gmp{+}\gnl
\gcl{1}\gnl
\gob{1}{H}
\gend
\quad
\text{and}
\quad
S\inv \,
=
\gbeg{1}{5}
\got{1}{H}\gnl
\gcl{1}\gnl
\gmp{-}\gnl
\gcl{1}\gnl
\gob{1}{H}
\gend
$$

Using graphical calculus, \eqref{eqantipodedef} obtains the following form:
\begin{align}\label{eqantipode}
\gbeg{2}{5}
\got{2}{H}\gnl
\gcmu\gnl
\gmp{+}\gcl{1}\gnl
\gmu\gnl
\gob{2}{H}
\gend
=
\gbeg{2}{5}
\got{2}{H}\gnl
\gcmu\gnl
\gcl{1}\gmp{+}\gnl
\gmu\gnl
\gob{2}{H}
\gend
=
\gbeg{1}{4}
\got{1}{H}\gnl
\gcu{1}\gnl
\gu{1}\gnl
\gob{1}{H}
\gend
\end{align}	

Let $A$ be an algebra in $\C$; for a left $A$-module $M$ in $\C$ respectively a right $A$-module 
$M$ in $\C$, we use the following graphical notation:
$$
\mu^- \,
=
\gbeg{2}{3}
\got{1}{A}\got{1}{M}\gnl
\glm\gnl
\gvac{1}\gob{1}{M}
\gend
\quad
\text{resp.}
\quad
\mu^+ \,
=
\gbeg{2}{3}
\got{1}{M}\got{1}{A}\gnl
\grm\gnl
\gob{1}{M}
\gend
$$

In the sequel, $H$ will denote a braided Hopf algebra with bijective antipode. 
The subcategory of left $H$-modules in $\C$, denoted by ${}_H\C$, is monoidal. 
If $M, N \in {}_H\C$, then $M \otimes N \in {}_H\C$ via the diagonal action:
\begin{align}\label{eqdiagonalaction}
\gbeg{3}{4}
\got{1}{H}\gvac{1}\got{1}{M\otimes N}\gnl
\gcn{2}{1}{1}{3}\gcl{1}\gnl
\gvac{1}\glm\gnl
\gvac{2}\gob{1}{M\otimes N}
\gend
=
\gbeg{4}{5}
\got{2}{H}\got{1}{M}\got{1}{N}\gnl
\gcmu\gcl{1}\gcl{1}\gnl
\gcl{1}\gbr\gcl{1}\gnl
\glm\glm\gnl
\gvac{1}\gob{1}{M}\gvac{1}\gob{1}{N}
\gend
\end{align}
Similarly, one can consider the (monoidal) category of right $H$-modules $\C_H$ and the (monoidal) 
category of $H$-bimodules ${}_H\C_H$. 

An object $A$ in $\C$ is said to be an $H$-module algebra if $A$ is an $H$-module and an algebra in $\C$, 
in such a way that its multiplication and unit are $H$-linear, i.e., $A$ is an algebra in the category ${}_H\C$. 
Graphically, we have:
\begin{align}\label{eqmodulealgebra}
\gbeg{3}{5}
\got{1}{H}\got{1}{A}\gvac{1}\got{1}{A}\gnl
\gcl{1}\gwmu{3}\gnl
\gcn{2}{1}{1}{3}\gcl{1}\gnl
\gvac{1}\glm\gnl
\gvac{2}\gob{1}{A}
\gend
=
\gbeg{4}{6}
\got{2}{H}\got{1}{A}\got{1}{A}\gnl
\gcmu\gcl{1}\gcl{1}\gnl
\gcl{1}\gbr\gcl{1}\gnl
\glm\glm\gnl
\gvac{1}\gwmu{3}\gnl
\gvac{2}\gob{1}{A}
\gend
\quad
\text{and}
\quad
\gbeg{2}{4}
\got{1}{H}\gnl
\gcl{1}\gu{1}\gnl
\glm\gnl
\gvac{1}\gob{1}{A}
\gend
=
\gbeg{2}{4}
\got{1}{H}\gnl
\gcl{1}\gu{1}\gnl
\gcu{1}\gcl{1}\gnl
\gvac{1}\gob{1}{A}
\gend
\end{align}

For an $H$-module algebra $A$, one can define the smash product $A \# H$ in $\C$. As an object, 
$A \# H = A\otimes H$, while the multiplication is given by:
$$
\nabla_{A\# H} =
\gbeg{5}{6}
\got{1}{A}\got{2}{H}\got{1}{A}\got{1}{H}\gnl
\gcl{1}\gcmu\gcl{1}\gcl{1}\gnl
\gcl{1}\gcl{1}\gbr\gcl{1}\gnl
\gcl{1}\glm\gcl{1}\gcl{1}\gnl
\gwmu{3}\gmu\gnl
\gob{3}{A}\gob{2}{H}
\gend
$$
The unit is given by $\eta_A \otimes \eta_H$. It is well-known (e.g. \cite{majidcross}) that in this way 
$A \# H$ becomes 
an algebra in $\C$.

For a left or respectively right  $H$-comodule $N$ in $\C$ we denote its $H$-comodule structure by:
$$
\lambda \,
=
\gbeg{2}{3}
\gvac{1}\got{1}{N}\gnl
\glcm\gnl
\gob{1}{H}\gob{1}{N}
\gend
\quad
\text{resp.}
\quad
\rho \,
=
\gbeg{2}{3}
\got{1}{N}\gnl
\grcm\gnl
\gob{1}{N}\gob{1}{H}
\gend
$$

Let ${}^H\C$ denote the category of left $H$-comodules in $\C$. Similar to \eqref{eqmodulealgebra}, 
now using the diagonal coaction, we obtain that the category ${}^H\C$ is monoidal. 
\begin{align}\label{eqdiagonalcoaction}
\gbeg{3}{4}
\gvac{2}\got{1}{M\otimes N}\gnl
\gvac{1}\glcm\gnl
\gcn{2}{1}{3}{1}\gcl{1}\gnl
\gob{1}{H}\gvac{1}\gob{1}{M\otimes N}
\gend
=
\gbeg{4}{5}
\gvac{1}\got{1}{M}\gvac{1}\got{1}{N}\gnl
\glcm\glcm\gnl
\gcl{1}\gbr\gcl{1}\gnl
\gmu\gcl{1}\gcl{1}\gnl
\gob{2}{H}\gob{1}{M}\gob{1}{N}
\gend
\end{align}

Thus we can consider algebras in ${}^H\C$, which we will call left $H$-comodule algebras. 
In other words, $A$ is a left $H$-comodule algebra if it satisfies:
\begin{align}\label{eqleftcomodulealgebra}
\gbeg{3}{4}
\got{1}{A}\gvac{1}\got{1}{A}\gnl
\gwmu{3}\gnl
\glcm\gnl
\gob{1}{H}\gob{1}{A}
\gend
=
\gbeg{4}{5}
\gvac{1}\got{1}{A}\gvac{1}\got{1}{A}\gnl
\glcm\glcm\gnl
\gcl{1}\gbr\gcl{1}\gnl
\gmu\gmu\gnl
\gob{2}{H}\gob{2}{A}
\gend
\quad
\text{and}
\quad
\gbeg{2}{3}
\gvac{1}\gu{1}\gnl
\glcm\gnl
\gob{1}{H}\gob{1}{A}
\gend
=
\gbeg{2}{3}
\gu{1}\gu{1}\gnl
\gcl{1}\gcl{1}\gnl
\gob{1}{H}\gob{1}{A}
\gend
\end{align}
Similarly, we can consider the category of right comodules $\C^H$ as well as the notion of a 
right $H$-comodule algebra. Specifically, a right $H$-comodule algebra $A$ is an algebra in 
$\C^H$ and satisfies:
\begin{align}\label{eqrightcomodulealgebra}
\gbeg{3}{4}
\got{1}{A}\gvac{1}\got{1}{A}\gnl
\gwmu{3}\gnl
\gvac{1}\grcm\gnl
\gvac{1}\gob{1}{A}\gob{1}{H}
\gend
=
\gbeg{4}{5}
\got{1}{A}\gvac{1}\got{1}{A}\gnl
\grcm\grcm\gnl
\gcl{1}\gbr\gcl{1}\gnl
\gmu\gmu\gnl
\gob{2}{A}\gob{2}{H}
\gend
\quad
\text{and}
\quad
\gbeg{2}{3}
\gu{1}\gnl
\grcm\gnl
\gob{1}{A}\gob{1}{H}
\gend
=
\gbeg{2}{3}
\gu{1}\gu{1}\gnl
\gcl{1}\gcl{1}\gnl
\gob{1}{A}\gob{1}{H}
\gend
\end{align}

An $H$-bicomodule algebra is a left and right $H$-comodule algebra satisfying the 
additional bicomodule relation:
$$
\gbeg{3}{4}
\gvac{1}\got{1}{A}\gnl
\glcm\gnl
\gcl{1}\grcm\gnl
\gob{1}{H}\gob{1}{A}\gob{1}{H}
\gend
=
\gbeg{3}{4}
\gvac{1}\got{1}{A}\gnl
\gvac{1}\grcm\gnl
\glcm\gcl{1}\gnl
\gob{1}{H}\gob{1}{A}\gob{1}{H}
\gend
$$

Let us recall the definition of braided Yetter-Drinfeld modules, also called crossed modules, 
as introduced by Bespalov in \cite{bespalov}.
\begin{definition}
A braided left-left Yetter-Drinfeld module $(M,\mu^-,\lambda)$ over $H$ in the category $\C$ is 
simultaneously a left $H$-module and a left $H$-comodule in $\C$ satisfying the following 
compatibility relation:
\begin{align}\label{eqyd}
\gbeg{4}{5}
\got{2}{H}\gvac{1}\got{1}{M}\gnl
\gcmu\glcm\gnl
\gcl{1}\gbr\gcl{1}\gnl
\gmu\glm\gnl
\gob{2}{H}\gvac{1}\gob{1}{M}
\gend
=
\gbeg{3}{8}
\got{2}{H}\got{1}{M}\gnl
\gcmu\gcl{1}\gnl
\gcl{1}\gbr\gnl
\glm\gcl{1}\gnl
\glcm\gcl{1}\gnl
\gcl{1}\gbr\gnl
\gmu\gcl{1}\gnl
\gob{2}{H}\gob{1}{M}
\gend
\end{align}
Denote by $\ydc$ the category of left-left Yetter-Drinfeld modules in $\C$ (the morphisms are left 
$H$-module and left $H$-comodule morphisms in $\C$).
\end{definition}

Let $M,N \in \ydc$; using the diagonal action \eqref{eqdiagonalaction} and the diagonal coaction \eqref{eqdiagonalcoaction}, $M \otimes N$ becomes a left-left Yetter-Drinfeld module as well. 
Hence the category $\ydc$ is monoidal. It is also (pre) braided, the braiding and its inverse are given by:
$$
c_{M,N} =
\gbeg{3}{5}
\gvac{1}\got{1}{M}\got{1}{N}\gnl
\glcm\gcl{1}\gnl
\gcl{1}\gbr\gnl
\glm\gcl{1}\gnl
\gvac{1}\gob{1}{N}\gob{1}{M}
\gend
\quad
\text{and}
\quad
c_{M,N}\inv =
\gbeg{3}{7}
\got{1}{N}\gvac{1}\got{1}{M}\gnl
\gcl{1}\glcm\gnl
\gcl{1}\gibr\gnl
\gibr\gmp{-}\gnl
\gcl{1}\gibr\gnl
\gcl{1}\glm\gnl
\gob{1}{M}\gvac{1}\gob{1}{N}
\gend
$$

An algebra $A$ in $\ydc$ is called a (braided) Yetter-Drinfeld module algebra. Equivalently, 
$A\in \ydc$ is a Yetter-Drinfeld module algebra if it is a left $H$-module algebra and a left $H$-comodule algebra.
\begin{proposition}\label{propsmashbicomodule}
Let $(\C,\otimes,\phi)$ be a braided monoidal category and $H \in \C$ a braided Hopf algebra. 
Assume that $A$ is a left-left Yetter-Drinfeld module algebra. Then the smash product algebra $A \# H$ is an 
$H$-bicomodule algebra in the category $\C$, the structures being given by:
$$
\lambda_{A \# H} =
\gbeg{4}{5}
\gvac{1}\got{3}{A \; \# \; H}\gnl
\glcm\gcmu\gnl
\gcl{1}\gbr\gcl{1}\gnl
\gmu\gcl{1}\gcl{1}\gnl
\gob{2}{H}\gob{2}{A \# H}
\gend
\quad
\text{and}
\quad
\rho_{A \# H} =
\gbeg{3}{3}
\got{3}{A \; \# \; H}\gnl
\gcl{1}\gcmu\gnl
\gob{2}{A \# H}\gob{1}{H}
\gend
$$
Moreover, the morphism $\eta_A \otimes H : H \rightarrow A \# H$ is a morphism of 
$H$-bicomodule algebras.
\end{proposition}
\begin{proof} Obviously, by the coassociativity of $\Delta$, $A \# H$ is a right $H$-comodule. 
$A \# H$ is easily seen to be a left $H$-comodule as well, since
\begin{align*}
&(\Delta \otimes A \# H) \circ \lambda_{A \# H} =	\\
&\gbeg{4}{7}
\gvac{1}\got{1}{A}\got{2}{H}\gnl
\glcm\gcmu\gnl
\gcl{1}\gbr\gcl{1}\gnl
\gmu\gcl{1}\gcl{1}\gnl
\gcmu\gcl{1}\gcl{1}\gnl
\gob{1}{H}\gob{1}{H}\gob{1}{A}\gob{1}{H}
\gend
\overset{\eqref{eqbialgebra}}
=
\gbeg{6}{8}
\gvac{2}\got{1}{A}\gvac{1}\got{1}{H}\gnl
\gvac{1}\glcm\gwcm{3}\gnl
\gcn{2}{1}{3}{2}\gbr\gvac{1}\gcl{1}\gnl
\gcn{2}{1}{2}{2}\gcn{1}{1}{1}{2}\gcn{2}{1}{1}{3}\gcl{1}\gnl
\gcmu\gcmu\gcl{1}\gcl{1}\gnl
\gcl{1}\gbr\gcl{1}\gcl{1}\gcl{1}\gnl
\gmu\gmu\gcl{1}\gcl{1}\gnl
\gob{2}{H}\gob{2}{H}\gob{1}{A}\gob{1}{H}
\gend
\overset{comod.}{\underset{nat.}{=}}
\gbeg{6}{7}
\gvac{2}\got{1}{A}\gvac{1}\got{2}{H}\gnl
\gvac{1}\glcm\gwcmh{4}{2}{5}\gnl
\gcn{1}{1}{3}{1}\glcm\gcmu\gcl{1}\gnl
\gcl{1}\gcl{1}\gbr\gcl{1}\gcl{1}\gnl
\gcl{1}\gbr\gbr\gcl{1}\gnl
\gmu\gmu\gcl{1}\gcl{1}\gnl
\gob{2}{H}\gob{2}{H}\gob{1}{A}\gob{1}{H}
\gend
\overset{coasso.}{\underset{nat.}{=}}
\gbeg{6}{8}
\gvac{1}\got{1}{A}\gvac{1}\got{1}{H}\gnl
\glcm\gwcm{3}\gnl
\gcl{1}\gbr\gvac{1}\gcl{1}\gnl
\gmu\gcn{2}{1}{1}{3}\gcn{1}{1}{1}{2}\gnl
\gcn{1}{1}{2}{2}\gvac{1}\glcm\gcmu\gnl
\gcn{1}{1}{2}{2}\gvac{1}\gcl{1}\gbr\gcl{1}\gnl
\gcn{1}{1}{2}{2}\gvac{1}\gmu\gcl{1}\gcl{1}\gnl
\gob{2}{H}\gob{2}{H}\gob{1}{A}\gob{1}{H}
\gend
\\
&=(H \otimes \lambda_{A \# H}) \circ \lambda_{A \# H}
\end{align*}
Again by coassociativity $A\#H$ becomes an $H$-bicomodule. As it is straightforward to verify that 
$A\# H$ becomes a right $H$-comodule algebra, the only thing that remains to be proved is the fact that 
$A \# H$ is a left $H$-comodule algebra. We verify:
\begin{align*}
&\lambda_{A \# H} \circ \nabla_{A \# H} \\
&=	
\gbeg{5}{9}
\got{1}{A}\got{2}{H}\got{1}{A}\got{1}{H}\gnl
\gcl{1}\gcmu\gcl{1}\gcl{1}\gnl
\gcl{1}\gcl{1}\gbr\gcl{1}\gnl
\gcl{1}\glm\gcl{1}\gcl{1}\gnl
\gwmu{3}\gmu\gnl
\glcm\gwcmh{4}{1}{5}\gnl
\gcl{1}\gbr\gvac{1}\gcl{1}\gnl
\gmu\gcl{1}\gvac{1}\gcl{1}\gnl
\gob{2}{A}\gob{1}{H}\gvac{1}\gob{1}{H}
\gend
\overset{\eqref{eqbialgebra}}{\underset{\eqref{eqleftcomodulealgebra}}{=}}
\gbeg{8}{11}
\gvac{1}\got{1}{A}\got{2}{H}\got{1}{A}\gvac{1}\got{1}{H}\gnl
\gvac{1}\gcl{1}\gcmu\gcl{1}\gvac{1}\gcl{1}\gnl
\gvac{1}\gcl{1}\gcl{1}\gbr\gvac{1}\gcl{1}\gnl
\gvac{1}\gcl{1}\glm\gcn{1}{1}{1}{2}\gvac{1}\gcn{1}{1}{1}{2}\gnl
\glcm\glcm\gcmu\gcmu\gnl
\gcl{1}\gbr\gcl{1}\gcl{1}\gbr\gcl{1}\gnl
\gmu\gmu\gmu\gmu\gnl
\gcn{2}{1}{2}{2}\gcn{2}{1}{2}{3}\gcn{2}{1}{2}{1}\gcn{2}{1}{2}{2}\gnl
\gcn{2}{1}{2}{2}\gvac{1}\gbr\gvac{1}\gcn{2}{1}{2}{2}\gnl
\gwmuh{4}{2}{7}\gcl{1}\gvac{1}\gcn{1}{1}{2}{2}\gnl
\gvac{1}\gob{2}{H}\gvac{1}\gob{1}{A}\gvac{1}\gob{2}{H}
\gend
\overset{coasso.}{\underset{nat.}{=}}
\gbeg{8}{15}
\gvac{1}\got{1}{A}\gvac{1}\got{1}{H}\gvac{1}\got{1}{A}\got{2}{H}\gnl
\gvac{1}\gcl{1}\gwcmh{3}{2}{5}\gcl{1}\gcn{1}{1}{2}{2}\gnl
\gvac{1}\gcl{1}\gcmu\gbr\gcmu\gnl
\gvac{1}\gcl{1}\gcl{1}\gbr\gbr\gcl{1}\gnl
\gvac{1}\gcl{1}\glm\gcl{1}\gcl{1}\gmu\gnl
\glcm\glcm\gcl{1}\gcl{1}\gcn{1}{1}{2}{2}\gnl
\gcl{1}\gcl{1}\gcl{1}\gbr\gcl{1}\gcn{1}{1}{2}{2}\gnl
\gcl{1}\gcl{1}\gmu\gcl{1}\gcl{1}\gcn{1}{1}{2}{2}\gnl
\gcl{1}\gcl{1}\gcn{2}{1}{2}{1}\gcl{1}\gcl{1}\gcn{1}{1}{2}{2}\gnl
\gcl{1}\gbr\gvac{1}\gcl{1}\gcl{1}\gcn{1}{1}{2}{2}\gnl
\gmu\gwmuh{3}{1}{5}\gcn{1}{1}{1}{1}\gcn{1}{1}{2}{2}\gnl
\gcn{2}{1}{2}{2}\gvac{1}\gcn{1}{1}{1}{1}\gcn{2}{1}{3}{1}\gcn{1}{1}{2}{2}\gnl
\gcn{2}{1}{2}{2}\gvac{1}\gbr\gvac{1}\gcn{1}{1}{2}{2}\gnl
\gwmuh{4}{2}{7}\gcl{1}\gvac{1}\gcn{1}{1}{2}{2}\gnl
\gob{4}{H}\gob{1}{A}\gvac{1}\gob{2}{H}
\gend
\\
&\overset{\eqref{eqyd}}{=}
\gbeg{9}{13}
\gvac{1}\got{1}{A}\gvac{1}\got{2}{H}\gvac{1}\got{1}{A}\got{2}{H}\gnl
\gvac{1}\gcl{1}\gwcmh{4}{2}{7}\gcl{1}\gcn{1}{1}{2}{2}\gnl
\gvac{1}\gcl{1}\gcn{2}{1}{2}{2}\gvac{1}\gbr\gcmu\gnl
\gvac{1}\gcl{1}\gcmu\glcm\gbr\gcl{1}\gnl
\gvac{1}\gcl{1}\gcl{1}\gbr\gcl{1}\gcl{1}\gmu\gnl
\glcm\gmu\glm\gcl{1}\gcn{1}{1}{2}{2}\gnl
\gcl{1}\gcl{1}\gcn{2}{1}{2}{1}\gvac{1}\gcl{1}\gcl{1}\gcn{1}{1}{2}{2}\gnl
\gcl{1}\gbr\gvac{2}\gcl{1}\gcl{1}\gcn{1}{1}{2}{2}\gnl
\gmu\gwmuh{4}{1}{7}\gcl{1}\gcn{1}{1}{2}{2}\gnl
\gcn{1}{1}{2}{2}\gvac{2}\gcn{2}{1}{2}{3}\gcn{2}{1}{3}{1}\gcn{1}{1}{2}{2}\gnl
\gcn{1}{1}{2}{2}\gvac{3}\gbr\gvac{1}\gcn{1}{1}{2}{2}\gnl
\gwmuh{5}{2}{9}\gcl{1}\gvac{1}\gcn{1}{1}{2}{2}\gnl
\gob{5}{H}\gob{1}{A}\gvac{1}\gob{2}{H}
\gend
\overset{(co)asso.}{\underset{nat.}{=}}
\gbeg{9}{12}
\gvac{1}\got{1}{A}\got{2}{H}\gvac{2}\got{1}{A}\got{2}{H}\gnl
\glcm\gwcmh{2}{1}{4}\gvac{1}\glcm\gcmu\gnl
\gcl{1}\gbr\gcmu\gcl{1}\gcl{1}\gcl{1}\gcl{1}\gnl
\gmu\gcl{1}\gcl{1}\gbr\gcl{1}\gcl{1}\gcl{1}\gnl
\gcn{2}{1}{2}{2}\gcl{1}\gbr\gbr\gcl{1}\gcl{1}\gnl
\gcn{2}{1}{2}{2}\gbr\glm\gbr\gcl{1}\gnl
\gcn{2}{1}{2}{2}\gcl{1}\gwmu{3}\gcl{1}\gmu\gnl
\gcn{2}{1}{2}{2}\gcl{1}\gvac{1}\gcn{2}{1}{1}{3}\gcl{1}\gcn{1}{1}{2}{2}\gnl
\gcn{2}{1}{2}{2}\gcl{1}\gvac{2}\gbr\gcn{1}{1}{2}{2}\gnl
\gcn{2}{1}{2}{2}\gwmuh{4}{1}{7}\gcl{1}\gcn{1}{1}{2}{2}\gnl
\gwmuh{5}{2}{8}\gvac{1}\gcl{1}\gcn{1}{1}{2}{2}\gnl
\gob{5}{H}\gvac{1}\gob{1}{A}\gob{2}{H}
\gend
\overset{nat.}{=}
\gbeg{8}{12}
\gvac{1}\got{1}{A}\got{2}{H}\gvac{1}\got{1}{A}\got{2}{H}\gnl
\glcm\gcmu\glcm\gcmu\gnl
\gcl{1}\gbr\gcl{1}\gcl{1}\gbr\gcl{1}\gnl
\gmu\gcl{1}\gcl{1}\gmu\gcl{1}\gcl{1}\gnl
\gcn{2}{1}{2}{2}\gcl{1}\gcl{1}\gcn{2}{1}{2}{1}\gcl{1}\gcl{1}\gnl
\gcn{2}{1}{2}{2}\gcl{1}\gbr\gvac{1}\gcl{1}\gcl{1}\gnl
\gcn{2}{1}{2}{2}\gbr\gcn{2}{1}{1}{2}\gcl{1}\gcl{1}\gnl
\gwmuh{3}{2}{5}\gcl{1}\gcmu\gcl{1}\gcl{1}\gnl
\gvac{1}\gcl{1}\gvac{1}\gcl{1}\gcl{1}\gbr\gcl{1}\gnl
\gvac{1}\gcl{1}\gvac{1}\gcl{1}\glm\gcl{1}\gcl{1}\gnl
\gvac{1}\gcl{1}\gvac{1}\gwmuh{3}{1}{5}\gmu\gnl
\gvac{1}\gob{1}{H}\gvac{1}\gob{3}{A}\gob{2}{H}
\gend
\\
&=\nabla_{A \# H} \circ (\lambda_{A \# H}\otimes \lambda_{A \# H})
\end{align*}
The last statement in the proposition is easily verified and is left to the reader.
\end{proof}

The aim of the remaining part of this section is to prove a converse of Proposition \ref{propsmashbicomodule}. 
For this we will need the concept of split idempotents.
\begin{definition}
An idempotent $e: X \rightarrow X$ in the category $\C$ is said to be split if there exists an object 
$X_e \in \C$ and morphisms $i_e : X_e \rightarrow X$ and $p_e : X \rightarrow X_e$ such that 
$p_e \circ i_e = id_{X_e}$ and $i_e \circ p_e = e$. We say that $\C$ admits split idempotents if 
any idempotent in $\C$ is split.
\end{definition}
Note that any category $\C$ can be embedded in a category $\hat{\C}$, also denoted $Split(\C)$, 
which admits split idempotents. $\hat{\C}$ is called the Karoubi enveloping category of $\C$. If 
$\C$ is (braided) monoidal, so is $\hat{\C}$ (cf. \cite{lyubashenko}).

From now on, we will assume that the braided monoidal category $\C$ admits split idempotents. 
Moreover, splittings are chosen as described in \cite[Appendix A]{bespalovdrabant}.

Let us recall the definition of (two-fold) Hopf (bi-)modules from \cite{bespalovdrabant}.
\begin{definition}
$B$ is called a right-right $H$-Hopf module in $\C$ if $B$ is a right $H$-module and a right 
$H$-comodule such that:
$$
\gbeg{2}{5}
\got{1}{B}\got{1}{H}\gnl
\grm\gnl
\gcl{1}\gnl
\grcm\gnl
\gob{1}{B}\gob{1}{H}
\gend
=
\gbeg{4}{5}
\got{1}{B}\gvac{1}\got{2}{H}\gnl
\grcm\gcmu\gnl
\gcl{1}\gbr\gcl{1}\gnl
\grm\gmu\gnl
\gob{1}{B}\gvac{1}\gob{2}{H}
\gend
$$

In other words, the right $B$-action is $H$-colinear. Similarly, one can define left-right $H$-Hopf modules, i.e. 
$B$ then is a left $H$-module in the category $\C^H$.

A two-fold Hopf module $B$ is an object in $\C$ which is at the same time a left-right and a right-right 
$H$-Hopf module, or equivalently, $B$ is an $H$-bimodule in $\C^H$. The category of two-fold Hopf 
modules together with $H$-bilinear and $H$-colinear morphisms is denoted by ${}_H\C_H^H$.

Finally, $B$ is said to be an $H$-Hopf bimodule if $B$ is an $H$-bimodule in the monoidal category 
${}^H\C^H$. Let $\CC$ denote the category of $H$-Hopf bimodules together with $H$-bilinear 
$H$-bicolinear morphisms.
\end{definition}

If $B$ is a right $H$-Hopf module, one can consider the morphism $E : B \rightarrow B$ defined by 
$E = \mu^+ \circ (B \otimes S) \circ \rho$, that is:
\begin{align}\label{defe}
E =
\gbeg{2}{5}
\got{1}{B}\gnl
\grcm\gnl
\gcl{1}\gmp{+}\gnl
\grm\gnl
\gob{1}{B}
\gend
\end{align}
Then, by \cite[Proposition 3.2.1]{bespalovdrabant}, $E$ is an idempotent. By assumption, 
there exist an object $B_0 \in \C$ and morphisms $i : B_0 \rightarrow B$ and 
$p : B \rightarrow B_0$ such that:
\begin{align}\label{eqip}
  \begin{aligned}
  &p \circ i = id_{B_0}  \\
  &i \circ p = E
  \end{aligned}
\end{align}
In addition, it is shown in \cite{bespalovdrabant} that $(B_0,i)$ is the equalizer of $\rho$ and 
$B\otimes \eta_H$. In other words, $B_0$ is equal to the coinvariants subobject $B^{coH}$. 
Using graphical calculus we obtain:
\begin{align}\label{iequalizer}
\gbeg{2}{4}
\got{1}{B_0}\gnl
\gbmp{i}\gnl
\grcm\gnl
\gob{1}{B}\gob{1}{H}
\gend
=
\gbeg{2}{4}
\got{1}{B_0}\gnl
\gbmp{i}\gnl
\gcl{1}\gu{1}\gnl
\gob{1}{B}\gob{1}{H}
\gend
\end{align}
$(B_0,p)$ is at the same time also equal to the coequalizer of $\mu^+$ and $B \otimes \varepsilon_H$, 
that is $B_0$ is the object of $H$-invariants. In particular, we have:
\begin{align}\label{pcoequalizer}
\gbeg{3}{5}
\got{1}{B_0}\gvac{1}\got{1}{H}\gnl
\gcl{1}\gvac{1}\gbmp{v}\gnl
\gwmu{3}\gnl
\gvac{1}\gbmp{p}\gnl
\gvac{1}\gob{1}{B}
\gend
=
\gbeg{2}{5}
\got{1}{B_0}\got{1}{H}\gnl
\gcl{1}\gcl{1}\gnl
\gbmp{p}\gcu{1}\gnl
\gcl{1}\gnl
\gob{1}{B}
\gend
\end{align}

If $B$ is a two-fold Hopf module, one can consider the adjoint $H$-action on $B$:
\begin{align*}
ad =
\gbeg{3}{6}
\got{2}{H}\got{1}{B}\gnl
\gcmu\gcl{1}\gnl
\gcl{1}\gbr\gnl
\glm\gmp{+}\gnl
\gvac{1}\grm\gnl
\gvac{1}\gob{1}{B}
\gend
\end{align*}
By \cite[Proposition 3.6.2]{bespalovdrabant} we have:
\begin{align}\label{eqead}
  \begin{aligned}
E \circ ad &= E \circ \mu^- = ad \circ (H \otimes E)	\\
p \circ ad &= p \circ \mu^-
  \end{aligned}
\end{align}
This allows us to define a left $H$-module structure on $B_0$, say $ad_0$, as follows:
\begin{align}\label{eqbzero}
ad_0
=
\gbeg{2}{7}
\got{1}{H}\got{1}{B_0}\gnl
\gcl{1}\gbmp{i}\gnl
\gcl{1}\gcl{1}\gnl
\gnotc{ad}\glmpt\grmptb\gnl
\gvac{1}\gcl{1}\gnl
\gvac{1}\gbmp{p}\gnl
\gvac{1}\gob{1}{B_0}
\gend
=
\gbeg{3}{7}
\got{2}{H}\got{1}{B_0}\gnl
\gcmu\gbmp{i}\gnl
\gcl{1}\gbr\gnl
\glm\gmp{+}\gnl
\gvac{1}\grm\gnl
\gvac{1}\gbmp{p}\gnl
\gvac{1}\gob{1}{B_0}
\gend
\end{align}
By construction, we have:
\begin{align}\label{eqiad}
i \circ ad_0 = ad \circ (H \otimes i)
\end{align}

The following is a (partial) generalization of \cite[Proposition 1.2]{alvarezvilaboa}, where a stronger 
condition (existence of (co)equalizers in $\C$) is assumed.
\begin{proposition}\label{proprightcomod}
Let $(\C,\otimes,\phi)$ be a braided monoidal category and $H \in \C$ a braided Hopf algebra. Let 
$B$ be a right $H$-comodule algebra in $\C$ such that there exists an $H$-colinear algebra morphism 
$v : H \rightarrow B$. We can consider $B_0$ as above, which now is an $H$-module algebra. 
Furthermore $B \cong B_0 \# H$ as right $H$-comodule algebras.
\end{proposition}
\begin{proof}
$B$ becomes a two-fold $H$-Hopf module via
\begin{align}\label{bistwofold}
\gbeg{2}{3}
\got{1}{H}\got{1}{B}\gnl
\glm\gnl
\gvac{1}\gob{1}{B}
\gend
=
\gbeg{2}{4}
\got{1}{H}\got{1}{B}\gnl
\gbmp{v}\gcl{1}\gnl
\gmu\gnl
\gob{2}{B}
\gend
\quad
\text{and}
\quad
\gbeg{2}{3}
\got{1}{B}\got{1}{H}\gnl
\grm\gnl
\gob{1}{B}
\gend
=
\gbeg{2}{4}
\got{1}{B}\got{1}{H}\gnl
\gcl{1}\gbmp{v}\gnl
\gmu\gnl
\gob{2}{B}
\gend
\end{align}
For example:
$$
\gbeg{2}{5}
\got{1}{B}\got{1}{H}\gnl
\grm\gnl
\gcl{1}\gnl
\grcm\gnl
\gob{1}{B}\gob{1}{H}
\gend
\overset{\eqref{bistwofold}}=
\gbeg{3}{5}
\got{1}{B}\gvac{1}\got{1}{H}\gnl
\gcl{1}\gvac{1}\gbmp{v}\gnl
\gwmu{3}\gnl
\gvac{1}\grcm\gnl
\gvac{1}\gob{1}{B}\gob{1}{H}
\gend
\overset{\eqref{eqrightcomodulealgebra}}=
\gbeg{4}{6}
\got{1}{B}\gvac{1}\got{1}{H}\gnl
\gcl{1}\gvac{1}\gbmp{v}\gnl
\grcm\grcm\gnl
\gcl{1}\gbr\gcl{1}\gnl
\gmu\gmu\gnl
\gob{2}{B}\gob{2}{H}
\gend
=
\gbeg{4}{6}
\got{1}{B}\gvac{1}\got{2}{H}\gnl
\grcm\gcmu\gnl
\gcl{1}\gcl{1}\gbmp{v}\gcl{1}\gnl
\gcl{1}\gbr\gcl{1}\gnl
\gmu\gmu\gnl
\gob{2}{B}\gob{2}{H}
\gend
\overset{\eqref{bistwofold}}=
\gbeg{4}{5}
\got{1}{B}\gvac{1}\got{2}{H}\gnl
\grcm\gcmu\gnl
\gcl{1}\gbr\gcl{1}\gnl
\grm\gmu\gnl
\gob{1}{B}\gvac{1}\gob{2}{H}
\gend
$$
where the third equality follows by right $H$-colinearity of $v$. Next, observe that:
\begin{align}\label{eqbmodalg}
\gbeg{4}{6}
\got{2}{H}\got{1}{B}\got{1}{B}\gnl
\gcmu\gcl{1}\gcl{1}\gnl
\gcl{1}\gbr\gcl{1}\gnl
\gnotc{ad}\glmpt\grmptb\gnotc{ad}\glmpt\grmptb\gnl
\gvac{1}\gwmu{3}\gnl
\gvac{2}\gob{1}{B}
\gend
=
\gbeg{6}{11}
\gvac{1}\got{2}{H}\got{1}{B}\gvac{1}\got{1}{B}\gnl
\gvac{1}\gcmu\gcl{1}\gvac{1}\gcl{1}\gnl
\gcn{2}{1}{3}{2}\gbr\gvac{1}\gcl{1}\gnl
\gcn{2}{1}{2}{2}\gcl{1}\gcn{2}{1}{1}{2}\gcl{1}\gnl
\gcmu\gcl{1}\gcmu\gcl{1}\gnl
\gbmp{v}\gbr\gbmp{v}\gbr\gnl
\gmu\gmp{+}\gmu\gmp{+}\gnl
\gcn{2}{1}{2}{2}\gbmp{v}\gcn{2}{1}{2}{2}\gbmp{v}\gnl
\gwmuh{3}{2}{5}\gwmuh{3}{2}{5}\gnl
\gvac{1}\gwmuh{4}{1}{7}\gnl
\gvac{1}\gob{4}{B}
\gend
\overset{(co)asso.}{ \underset{nat.}{=}}
\gbeg{5}{10}
\got{2}{H}\got{1}{B}\gvac{1}\got{1}{B}\gnl
\gcmu\gcl{1}\gvac{1}\gcl{1}\gnl
\gcl{1}\gbr\gvac{1}\gcl{1}\gnl
\gbmp{v}\gcl{1}\gcn{2}{1}{1}{2}\gcl{1}\gnl
\gmu\gwcmh{2}{0}{3}\gcl{1}\gnl
\gcn{1}{1}{2}{1}\gcmu\gbr\gnl
\gcl{1}\gmp{+}\gcl{1}\gcl{1}\gmp{+}\gnl
\gcl{1}\gbmp{v}\gbmp{v}\gcl{1}\gbmp{v}\gnl
\gcl{1}\gmu\gmu\gnl
\gwmuh{3}{1}{4}\gcn{1}{1}{2}{2}\gnl
\gvac{1}\gwmuh{3}{1}{6}
\gend
\overset{\eqref{eqantipode}}{{=}}
\gbeg{5}{10}
\got{2}{H}\got{1}{B}\gvac{1}\got{1}{B}\gnl
\gcmu\gcl{1}\gvac{1}\gcl{1}\gnl
\gcl{1}\gbr\gvac{1}\gcl{1}\gnl
\gbmp{v}\gcl{1}\gcn{2}{1}{1}{2}\gcl{1}\gnl
\gmu\gwcmh{2}{1}{3}\gcl{1}\gnl
\gcn{1}{1}{2}{1}\gvac{1}\gcu{1}\gbr\gnl
\gcl{1}\gvac{1}\gvac{1}\gcl{1}\gmp{+}\gnl
\gcl{1}\gvac{1}\gu{1}\gcl{1}\gbmp{v}\gnl
\gcl{1}\gvac{1}\gbmp{v}\gmu\gnl
\gwmuh{3}{1}{5}\gcn{1}{1}{2}{2}\gnl
\gvac{1}\gwmuh{3}{1}{6}
\gend
\overset{asso.}{ \underset{nat.}{=}}
\gbeg{4}{7}
\got{2}{H}\got{1}{B}\got{1}{B}\gnl
\gcmu\gcl{1}\gcl{1}\gnl
\gcl{1}\gbr\gcl{1}\gnl
\gcl{1}\gcl{1}\gbr\gnl
\gcl{1}\gmu\gmp{+}\gnl
\gbmp{v}\gcn{2}{1}{2}{2}\gbmp{v}\gnl
\gwmuh{3}{1}{4}\gcl{1}\gnl
\gvac{1}\gwmuh{3}{1}{5}\gnl
\gend
=
\gbeg{3}{6}
\got{1}{H}\got{1}{B}\gvac{1}\got{1}{B}\gnl
\gcl{1}\gwmu{3}\gnl
\gcn{2}{1}{1}{3}\gcl{1}\gnl
\gvac{1}\gnotc{ad}\glmpt\grmptb\gnl
\gvac{2}\gcl{1}\gnl
\gvac{2}\gob{1}{B}
\gend
\end{align}
\\
Let $(B_0,i,p)$ be defined as above, with $H$-action on $B_0$ as in \eqref{eqbzero}. We verify that 
$B_0$ becomes an $H$-module algebra, hence we can consider the smash product algebra $B_0 \# H$. 
First, there is an algebra structure $\nabla_0 : B_0 \otimes B_0 \rightarrow B_0$ defined by 
$\nabla_0 = p \circ \nabla \circ (i\otimes i)$. Equivalently, using the fact that $(B_0,i)$ is the 
equalizer of $\rho$ and $B\otimes H$, $\nabla_0$ is uniquely defined by the relation: 
\begin{align}\label{eqinabla}
i \circ \nabla_0 = \nabla \circ (i \otimes i)
\end{align}
In order to show that $B_0$ is an $H$-module algebra, we have to show:
\begin{align*}
ad_0 \circ (H\otimes \nabla_0) =
\gbeg{4}{6}
\got{2}{H}\got{1}{B_0}\got{1}{B_0}\gnl
\gcmu\gcl{1}\gcl{1}\gnl
\gcl{1}\gbr\gcl{1}\gnl
\gnotc{ad_0}\glmpt\grmptb\gnotc{ad_0}\glmpt\grmptb\gnl
\gvac{1}\gwmu{3}\gnl
\gvac{2}\gob{1}{B_0}
\gend
\end{align*}
Now, the right hand side equals
\begin{eqnarray*}
&&\nabla_0 \circ (ad_0 \otimes ad_0) \circ (H \otimes \phi_{H,B_0} \otimes B_0) \circ (\Delta \otimes B_0 \otimes B_0)	\\
&\overset{def. \ \nabla_0}=&p \circ \nabla \circ (i\otimes i)\circ (ad_0 \otimes ad_0) \circ (H \otimes \phi_{H,B_0} \otimes B_0) \circ (\Delta \otimes B_0 \otimes B_0)	\\
&\overset{\eqref{eqiad}}=&p \circ \nabla \circ (ad \otimes ad)\circ (H \otimes i\otimes H \otimes i)\circ  (H \otimes \phi_{H,B_0} \otimes B_0) \circ (\Delta \otimes B_0 \otimes B_0)	\\
&\overset{nat.}=&p \circ \nabla \circ (ad \otimes ad)\circ (H \otimes \phi_{H,B} \otimes B)\circ (H \otimes H\otimes i \otimes i) \circ (\Delta \otimes B_0 \otimes B_0)	\\
&\overset{nat.}=&p \circ \nabla \circ (ad \otimes ad)\circ (H \otimes \phi_{H,B} \otimes B)\circ (\Delta \otimes B \otimes B) \circ (H \otimes i \otimes i)	\\
&\overset{\eqref{eqbmodalg}}=&p \circ ad \circ (H \otimes \nabla) \circ (H \otimes i \otimes i)	\\
&\overset{\eqref{eqinabla}}=&p \circ ad \circ (H \otimes i) \circ (H \otimes \nabla_0)	\\
&\overset{\eqref{eqiad}}=&p \circ i \circ ad_0 \circ (H \otimes \nabla_0)	\\
&\overset{\eqref{eqip}}=&ad_0 \circ (H \otimes \nabla_0)
\end{eqnarray*}
The verification that $\eta_{B_0}$ is left $H$-linear is left to the reader. Finally one can verify that $\omega : B_0 \# H \rightarrow B$,
\begin{align*}
\omega &=
\gbeg{2}{4}
\got{1}{B_0}\got{1}{H}\gnl
\gbmp{i}\gbmp{v}\gnl
\gmu\gnl
\gob{2}{B}
\gend
\intertext{is a right $H$-colinear algebra isomorphism, where}
\omega\inv &=
\gbeg{2}{4}
\got{1}{B}\gnl
\grcm\gnl
\gbmp{p}\gcl{1}\gnl
\gob{1}{B_0}\gob{1}{H}
\gend
\end{align*}
First, $\omega$ is right $H$-colinear since
\begin{align}\label{omegarightcolinear}
\gbeg{3}{5}
\got{1}{B_0}\gvac{1}\got{1}{H}\gnl
\gbmp{i}\gvac{1}\gbmp{v}\gnl
\gwmu{3}\gnl
\gvac{1}\grcm\gnl
\gvac{1}\gob{1}{B}\gob{1}{H}
\gend
\overset{\eqref{eqrightcomodulealgebra}}=
\gbeg{4}{6}
\got{1}{B_0}\gvac{1}\got{1}{H}\gnl
\gbmp{i}\gvac{1}\gbmp{v}\gnl
\grcm\grcm\gnl
\gcl{1}\gbr\gcl{1}\gnl
\gmu\gmu\gnl
\gob{2}{B}\gob{2}{H}
\gend
\overset{\eqref{iequalizer}}=
\gbeg{4}{6}
\got{1}{B_0}\gvac{1}\got{2}{H}\gnl
\gbmp{i}\gvac{1}\gcmu\gnl
\gcl{1}\gu{1}\gbmp{v}\gcl{1}\gnl
\gcl{1}\gbr\gcl{1}\gnl
\gmu\gmu\gnl
\gob{2}{B}\gob{2}{H}
\gend
=
\gbeg{3}{5}
\got{1}{B_0}\got{2}{H}\gnl
\gcl{1}\gcmu\gnl
\gbmp{i}\gbmp{v}\gcl{1}\gnl
\gmu\gcl{1}\gnl
\gob{2}{B}\gob{1}{H}
\gend
\end{align}
Using \eqref{omegarightcolinear}, we see that we have:
\begin{align*}
&\omega\inv\circ \omega
= 
\gbeg{3}{6}
\got{1}{B_0}\gvac{1}\got{1}{H}\gnl
\gbmp{i}\gvac{1}\gbmp{v}\gnl
\gwmu{3}\gnl
\gvac{1}\grcm\gnl
\gvac{1}\gbmp{p}\gcl{1}\gnl
\gvac{1}\gob{1}{B_0}\gob{1}{H}
\gend
\overset{\eqref{omegarightcolinear}}=
\gbeg{4}{6}
\got{1}{B_0}\gvac{1}\got{2}{H}\gnl
\gcl{1}\gvac{1}\gcmu\gnl
\gbmp{i}\gvac{1}\gbmp{v}\gcl{1}\gnl
\gwmu{3}\gcl{1}\gnl
\gvac{1}\gbmp{p}\gvac{1}\gcl{1}\gnl
\gvac{1}\gob{1}{B_0}\gvac{1}\gob{1}{H}
\gend
\overset{\eqref{pcoequalizer}}=
\gbeg{3}{6}
\got{1}{B_0}\got{2}{H}\gnl
\gcl{1}\gcmu\gnl
\gbmp{i}\gcl{1}\gcl{1}\gnl
\gbmp{p}\gcu{1}\gcl{1}\gnl
\gcl{1}\gvac{1}\gcl{1}\gnl
\gob{1}{B_0}\gvac{1}\gob{1}{H}
\gend
= id_{B_0 \# H}
\end{align*}
Likewise:
\begin{align*}
&\omega \circ \omega\inv
=
\gbeg{2}{6}
\got{1}{B}\gnl
\grcm\gnl
\gbmp{p}\gcl{1}\gnl
\gbmp{i}\gbmp{v}\gnl
\gmu\gnl
\gob{2}{B}
\gend
\overset{\eqref{eqip}}=
\gbeg{2}{5}
\got{1}{B}\gnl
\grcm\gnl
\gbmp{E}\gbmp{v}\gnl
\gmu\gnl
\gob{2}{B}
\gend
\overset{\eqref{defe}}=
\gbeg{3}{9}
\got{1}{B}\gnl
\grcm\gnl
\gcl{1}\gcn{1}{1}{1}{3}\gnl
\grcm\gcl{1}\gnl
\gcl{1}\gmp{+}\gcl{1}\gnl
\gcl{1}\gbmp{v}\gbmp{v}\gnl
\gmu\gcl{1}\gnl
\gwmuh{3}{2}{5}\gnl
\gvac{1}\gob{1}{B}
\gend
\overset{asso.}{=}
\gbeg{3}{9}
\got{1}{B}\gnl
\grcm\gnl
\gcl{1}\gcn{1}{1}{1}{2}\gnl
\gcl{1}\gcmu\gnl
\gcl{1}\gmp{+}\gcl{1}\gnl
\gcl{1}\gbmp{v}\gbmp{v}\gnl
\gcl{1}\gmu\gnl
\gwmuh{3}{1}{4}\gnl
\gvac{1}\gob{1}{B}
\gend
\overset{\eqref{eqantipode}}=
\gbeg{2}{7}
\got{1}{B}\gnl
\grcm\gnl
\gcl{1}\gcu{1}\gnl
\gcl{1}\gu{1}\gnl
\gcl{1}\gbmp{v}\gnl
\gmu\gnl
\gob{2}{B}
\gend
=
id_B
\end{align*}
Observe that we have:
\begin{align}\label{eqvi}
\gbeg{3}{6}
\got{1}{H}\gvac{1}\got{1}{B_0}\gnl
\gbmp{v}\gvac{1}\gbmp{i}\gnl
\gwmu{3}\gnl
\gvac{1}\gbmp{p}\gnl
\gvac{1}\gbmp{i}\gnl
\gvac{1}\gob{1}{B}
\gend
\overset{\eqref{eqip}}
=
\gbeg{3}{8}
\got{1}{H}\gvac{1}\got{1}{B_0}\gnl
\gbmp{v}\gvac{1}\gbmp{i}\gnl
\gwmu{3}\gnl
\gvac{1}\grcm\gnl
\gvac{1}\gcl{1}\gmp{+}\gnl
\gvac{1}\gcl{1}\gbmp{v}\gnl
\gvac{1}\gmu\gnl
\gvac{1}\gob{2}{B}
\gend
\overset{\eqref{eqrightcomodulealgebra}}=
\gbeg{4}{10}
\got{1}{H}\gvac{1}\got{1}{B_0}\gnl
\gbmp{v}\gvac{1}\gbmp{i}\gnl
\grcm\grcm\gnl
\gcl{1}\gbr\gcl{1}\gnl
\gmu\gmu\gnl
\gcn{2}{1}{2}{2}\gcn{2}{1}{2}{3}\gnl
\gcn{1}{1}{2}{2}\gvac{2}\gmp{+}\gnl
\gcn{1}{1}{2}{2}\gvac{2}\gbmp{v}\gnl
\gwmuh{4}{2}{7}\gnl
\gob{4}{B}
\gend
\overset{\eqref{iequalizer}}=
\gbeg{4}{10}
\got{2}{H}\got{1}{B_0}\gnl
\gcmu\gbmp{i}\gnl
\gbmp{v}\gcl{1}\gcl{1}\gu{1}\gnl
\gcl{1}\gbr\gcl{1}\gnl
\gmu\gmu\gnl
\gcn{2}{1}{2}{2}\gcn{2}{1}{2}{3}\gnl
\gcn{1}{1}{2}{2}\gvac{2}\gmp{+}\gnl
\gcn{1}{1}{2}{2}\gvac{2}\gbmp{v}\gnl
\gwmuh{4}{2}{7}\gnl
\gob{4}{B}
\gend
=
\gbeg{3}{8}
\got{2}{H}\got{1}{B_0}\gnl
\gcmu\gbmp{i}\gnl
\gbmp{v}\gcl{1}\gcl{1}\gnl
\gcl{1}\gbr\gnl
\gmu\gmp{+}\gnl
\gcn{1}{1}{2}{2}\gvac{1}\gbmp{v}\gnl
\gwmuh{3}{2}{5}\gnl
\gob{3}{B}
\gend
\end{align}
Finally, $\omega$ is an algebra morphism since:
\begin{align*}
&
\gbeg{5}{8}
\got{1}{B_0}\got{2	}{H}\got{1}{B_0}\got{2}{H}\gnl
\gcl{1}\gcmu\gcl{1}\gcn{1}{1}{2}{2}\gnl
\gcl{1}\gcl{1}\gbr\gcn{1}{1}{2}{2}\gnl
\gcl{1}\gnotc{ad_0}\glmpt\grmptb\gcl{1}\gcn{1}{1}{2}{2}\gnl
\gwmu{3}\gwmuh{3}{1}{4}\gnl
\gvac{1}\gbmp{i}\gvac{2}\gbmp{v}\gnl
\gvac{1}\gwmu{4}\gnl
\gvac{2}\gob{2}{B}
\gend
\overset{\eqref{eqbzero}}{\underset{\eqref{eqinabla}}{=}}
\gbeg{5}{9}
\got{1}{B_0}\got{2	}{H}\got{1}{B_0}\got{2}{H}\gnl
\gcl{1}\gcmu\gbmp{i}\gcn{1}{1}{2}{2}\gnl
\gcl{1}\gcl{1}\gbr\gcn{1}{1}{2}{2}\gnl
\gcl{1}\gnotc{ad}\glmpt\grmptb\gcl{1}\gcn{1}{1}{2}{2}\gnl
\gcl{1}\gvac{1}\gbmp{p}\gcl{1}\gcn{1}{1}{2}{2}\gnl
\gbmp{i}\gvac{1}\gbmp{i}\gwmuh{3}{1}{4}\gnl
\gwmu{3}\gvac{1}\gbmp{v}\gnl
\gvac{1}\gwmu{4}\gnl
\gvac{2}\gob{2}{B}
\gend
\overset{\eqref{eqip}}{\underset{\eqref{eqead}}{=}}
\gbeg{5}{9}
\got{1}{B_0}\got{2	}{H}\got{1}{B_0}\got{2}{H}\gnl
\gcl{1}\gcmu\gbmp{i}\gcn{1}{1}{2}{2}\gnl
\gcl{1}\gbmp{v}\gbr\gcn{1}{1}{2}{2}\gnl
\gcl{1}\gmu\gcl{1}\gcn{1}{1}{2}{2}\gnl
\gcl{1}\gcn{2}{1}{2}{3}\gcl{1}\gcn{1}{1}{2}{2}\gnl
\gbmp{i}\gvac{1}\gbmp{E}\gwmuh{3}{1}{4}\gnl
\gwmu{3}\gvac{1}\gbmp{v}\gnl
\gvac{1}\gwmu{4}\gnl
\gvac{2}\gob{2}{B}
\gend
\overset{\eqref{eqvi}}{\underset{}{=}}
\gbeg{6}{10}
\got{1}{B_0}\gvac{1}\got{1}{H}\gvac{1}\got{1}{B_0}\got{1}{H}\gnl
\gcl{1}\gwcmh{3}{2}{5}\gcl{1}\gcl{1}\gnl
\gcl{1}\gcmu\gbr\gcl{1}\gnl
\gcl{1}\gcl{1}\gbr\gbmp{v}\gbmp{v}\gnl
\gcl{1}\gbmp{v}\gbmp{i}\gmp{+}\gmu\gnl
\gcl{1}\gmu\gbmp{v}\gcn{1}{1}{2}{2}\gnl
\gbmp{i}\gwmuh{3}{2}{5}\gcn{1}{1}{2}{2}\gnl
\gwmu{3}\gvac{1}\gcn{1}{1}{2}{2}\gnl
\gvac{1}\gwmuh{5}{1}{8}\gnl
\gvac{3}\gob{1}{B}
\gend
\\
&\overset{(co)asso.}{\underset{nat.}{=}}
\gbeg{6}{12}
\got{1}{B_0}\got{2}{H}\got{1}{B_0}\gvac{1}\got{1}{H}\gnl
\gcl{1}\gcmu\gcl{1}\gvac{1}\gcl{1}\gnl
\gcl{1}\gcl{1}\gbr\gvac{1}\gcl{1}\gnl
\gcl{1}\gcl{1}\gcl{1}\gcn{1}{1}{1}{2}\gvac{1}\gcl{1}\gnl
\gcl{1}\gcl{1}\gcl{1}\gcmu\gcl{1}\gnl
\gbmp{i}\gbmp{v}\gbmp{i}\gmp{+}\gbmp{v}\gbmp{v}\gnl
\gmu\gcl{1}\gbmp{v}\gcl{1}\gcl{1}\gnl
\gcn{2}{1}{2}{2}\gcl{1}\gmu\gcl{1}\gnl
\gcn{2}{1}{2}{2}\gwmuh{3}{1}{4}\gcl{1}\gnl
\gcn{2}{1}{2}{2}\gvac{1}\gwmuh{3}{1}{5}\gnl
\gwmuh{5}{2}{9}\gnl
\gvac{2}\gob{1}{B}
\gend
\overset{\eqref{eqantipode}}{\underset{}{=}}
\gbeg{5}{10}
\got{1}{B_0}\got{2}{H}\got{1}{B_0}\got{1}{H}\gnl
\gcl{1}\gcmu\gcl{1}\gcl{1}\gnl
\gcl{1}\gcl{1}\gbr\gcl{1}\gnl
\gcl{1}\gcl{1}\gcl{1}\gcu{1}\gcl{1}\gnl
\gbmp{i}\gbmp{v}\gbmp{i}\gu{1}\gbmp{v}\gnl
\gmu\gcl{1}\gbmp{v}\gcl{1}\gnl
\gcn{2}{1}{2}{2}\gmu\gcl{1}\gnl
\gcn{2}{1}{2}{2}\gwmuh{3}{2}{5}\gnl
\gwmuh{4}{2}{7}\gnl
\gvac{1}\gob{2}{B}
\gend
=
\gbeg{4}{5}
\got{1}{B_0}\got{1}{H}\got{1}{B_0}\got{1}{H}\gnl
\gbmp{i}\gbmp{v}\gbmp{i}\gbmp{v}\gnl
\gmu\gmu\gnl
\gwmuh{4}{2}{6}\gnl
\gvac{1}\gob{2}{B}
\gend
\end{align*}
\end{proof}

The following is a mirror-symmetry version of \cite[Theorem 4.3.2]{bespalovdrabant}.
\begin{theorem}\label{thmydccc}
Let $(\C,\otimes,\phi)$ be a braided monoidal category which admits split idempotents and suppose 
$H \in \C$ is a braided Hopf algebra with bijective antipode. The categories $\ydc$ and $\CC$ are 
braided monoidal equivalent. In particular\\
(i) Let $V \in \ydc$, then $V \otimes H \in \CC$ via
$$
\mu^-_{V \otimes H} \,
=
\gbeg{4}{5}
\got{2}{H}\got{2}{V \otimes H}\gnl
\gcmu\gcl{1}\gcl{1}\gnl
\gcl{1}\gbr\gcl{1}\gnl
\glm\gmu\gnl
\gvac{2}\gob{1}{V \; \otimes \; H}
\gend
\quad
\text{and}
\quad
\mu^+_{V \otimes H} \,
=
\gbeg{3}{3}
\got{2}{V \otimes H}\got{1}{H}\gnl
\gcl{1}\gmu\gnl
\gob{3}{V \; \otimes \; H}
\gend
$$
$$
\lambda_{V \otimes H} \,
=
\gbeg{4}{5}
\gvac{1}\got{3}{V \; \otimes \; H}\gnl
\glcm\gcmu\gnl
\gcl{1}\gbr\gcl{1}\gnl
\gmu\gcl{1}\gcl{1}\gnl
\gob{2}{H}\gob{2}{V \otimes H}
\gend
\quad
\text{and}
\quad
\rho_{V \otimes H} \,
=
\gbeg{3}{3}
\got{3}{V \; \otimes \; H}\gnl
\gcl{1}\gcmu\gnl
\gob{2}{V \otimes H}\gob{1}{H}
\gend
$$
(ii) Let $M \in \CC$, then $M_0 \in \ydc$ with $H$-action induced by the adjoint action, similar as in \eqref{eqbzero}, and $H$-coaction inherited from $M$, that is the $H$-coaction is defined by the following relation:
\begin{align}\label{eqinherited}
\gbeg{2}{4}
\gvac{1}\got{1}{M_0}\gnl
\glcm\gnl
\gcl{1}\gbmp{i}\gnl
\gob{1}{H}\gob{1}{M}
\gend
=
\gbeg{2}{4}
\gvac{1}\got{1}{M_0}\gnl
\gvac{1}\gbmp{i}\gnl
\glcm\gnl
\gob{1}{H}\gob{1}{M}
\gend
\end{align}
\end{theorem}

We are now able to prove the structure theorem for braided bicomodule algebras.
\begin{theorem}
Let $(\C,\otimes,\phi)$ be a braided monoidal category which admits split idempotents and let $H \in \C$ 
be a braided Hopf algebra with bijective antipode. Assume $H$ is flat. Suppose $B$ is an $H$-bicomodule 
algebra such that there exists an $H$-bicolinear algebra morphism $v : H \rightarrow B$. Let $(B_0,i,p)$ 
be the splitting as in \eqref{eqip}, then $B_0 \in \ydc$ is a Yetter-Drinfeld module algebra. The morphism 
$\omega : B_0 \# H \rightarrow B$ of Proposition \ref{proprightcomod} becomes an isomorphism of 
$H$-bicomodule algebras.
\end{theorem}
\begin{proof}
First, $B$ becomes an object in $\CC$ via \eqref{bistwofold}. Hence, by Theorem \ref{thmydccc}, $B_0$ is 
an object in $\ydc$. On the other hand, by Proposition \ref{proprightcomod}, we know that $B_0$ is a left 
$H$-module algebra and that $\omega$ is a morphism of right $H$-comodule algebras.  Ergo, it suffices to 
show that $B_0$ is now also a left $H$-comodule algebra and that $\omega$ is left $H$-colinear as well. 
The first statement can be established as follows:
\begin{align*}
\gbeg{3}{5}
\got{1}{B_0}\gvac{1}\got{1}{B_0}\gnl
\gwmu{3}\gnl
\glcm\gnl
\gcl{1}\gbmp{i}\gnl
\gob{1}{H}\gob{1}{B}
\gend
\overset{\eqref{eqinherited}}
=
\gbeg{3}{5}
\got{1}{B_0}\gvac{1}\got{1}{B_0}\gnl
\gwmu{3}\gnl
\gvac{1}\gbmp{i}\gnl
\glcm\gnl
\gob{1}{H}\gob{1}{B}
\gend
\overset{\eqref{eqinabla}}
=
\gbeg{3}{5}
\got{1}{B_0}\gvac{1}\got{1}{B_0}\gnl
\gbmp{i}\gvac{1}\gbmp{i}\gnl
\gwmu{3}\gnl
\glcm\gnl
\gob{1}{H}\gob{1}{B}
\gend
\overset{\eqref{eqleftcomodulealgebra}}
=
\gbeg{4}{6}
\gvac{1}\got{1}{B_0}\gvac{1}\got{1}{B_0}\gnl
\gvac{1}\gbmp{i}\gvac{1}\gbmp{i}\gnl
\glcm\glcm\gnl
\gcl{1}\gbr\gcl{1}\gnl
\gmu\gmu\gnl
\gob{2}{H}\gob{2}{B}
\gend
\overset{\eqref{eqinherited}}
=
\gbeg{4}{6}
\gvac{1}\got{1}{B_0}\gvac{1}\got{1}{B_0}\gnl
\glcm\glcm\gnl
\gcl{1}\gbmp{i}\gcl{1}\gbmp{i}\gnl
\gcl{1}\gbr\gcl{1}\gnl
\gmu\gmu\gnl
\gob{2}{H}\gob{2}{B}
\gend
\overset{\eqref{eqinabla}}
=
\gbeg{4}{6}
\gvac{1}\got{1}{B_0}\gvac{1}\got{1}{B_0}\gnl
\glcm\glcm\gnl
\gcl{1}\gbr\gcl{1}\gnl
\gmu\gmu\gnl\gcn{2}{1}{2}{2}\gnotc{i}\glmp\grmpb\gnl
\gob{2}{H}\gvac{1}\gob{1}{B}
\gend
\end{align*}
Since $H$ is flat, the functor $H \otimes -$ preserves equalizers. Hence $H \otimes i$ is a monomorphism 
and by the above computation we may conclude that $B_0$ is a left $H$-comodule algebra (the fact that 
$\eta_{B_0}$ is $H$-colinear is trivial).
\\
To finish the proof, we verify that $\omega$ is also left $H$-colinear. Note that $B_0 \# H$ has the 
structure of a left $H$-comodule via
$$
\lambda_{B_0 \# H} =
\gbeg{4}{5}
\gvac{1}\got{3}{B_0 \; \# \; H}\gnl
\glcm\gcmu\gnl
\gcl{1}\gbr\gcl{1}\gnl
\gmu\gcl{1}\gcl{1}\gnl
\gob{2}{H}\gob{2}{B_0 \# H}
\gend
$$
as in Proposition \ref{propsmashbicomodule}. Then:
$$
(H \otimes \omega) \circ \lambda_{B_0 \# H} =
\gbeg{4}{6}
\gvac{1}\got{3}{B_0 \; \# \; H}\gnl
\glcm\gcmu\gnl
\gcl{1}\gbr\gcl{1}\gnl
\gmu\gcl{1}\gcl{1}\gnl
\gcn{2}{1}{2}{2}\gnotc{\omega}\glmpt\grmptb\gnl
\gob{2}{H}\gvac{1}\gob{1}{B}
\gend
=
\gbeg{4}{6}
\gvac{1}\got{3}{B_0 \; \# \; H}\gnl
\glcm\gcmu\gnl
\gcl{1}\gbmp{i}\gcl{1}\gbmp{v}\gnl
\gcl{1}\gbr\gcl{1}\gnl
\gmu\gmu\gnl
\gob{2}{H}\gob{2}{B}
\gend
\overset{\eqref{eqinherited}}{ \underset{(*)}{=}}
\gbeg{4}{6}
\gvac{1}\got{3}{B_0 \; \# \; H}\gnl
\gvac{1}\gbmp{i}\gvac{1}\gbmp{v}\gnl
\glcm\glcm\gnl
\gcl{1}\gbr\gcl{1}\gnl
\gmu\gmu\gnl
\gob{2}{H}\gob{2}{B}
\gend
\overset{\eqref{eqleftcomodulealgebra}}{=}
\gbeg{3}{5}
\got{1}{B_0}\gvac{1}\got{1}{H}\gnl
\gbmp{i}\gvac{1}\gbmp{v}\gnl
\gwmu{3}\gnl
\glcm\gnl
\gob{1}{H}\gob{1}{B}
\gend
=\lambda_B \circ \omega
$$
where in $(*)$ we have used the left $H$-colinearity of $v$. This concludes the proof.
\end{proof}

\end{document}